\documentclass[12pt]{amsart}
\usepackage{}
\usepackage{amsmath}
\usepackage{amsfonts}
\usepackage{amssymb}
\usepackage[all,cmtip]{xy}           
\usepackage{mathrsfs}
\usepackage{bbm}
\usepackage{stmaryrd}
\usepackage{bbding}
\usepackage{txfonts}
\usepackage[shortlabels]{enumitem}
\usepackage{ifpdf}
\ifpdf
\usepackage[colorlinks,final,backref=page,hyperindex]{hyperref}
\else
\usepackage[colorlinks,final,backref=page,hyperindex,hypertex]{hyperref}
\fi
\usepackage{amscd}
\usepackage{tikz-cd}
\usepackage[active]{srcltx}

\makeatletter

\newtheorem{theorem}{Theorem}[section]
\newtheorem{prop}[theorem]{Proposition}
\newtheorem{lem}[theorem]{Lemma}

\newtheorem{prop-def}{Proposition-Definition}[section]

\theoremstyle{definition}
\newtheorem{defn}[theorem]{Definition}

\newtheorem{remark}[theorem]{Remark}
\newtheorem{exam}[theorem]{Example}

\newcommand{\nc}{\newcommand}

\nc{\delete}[1]{{}}
\nc{\mmargin}[1]{}

\nc{\mlabel}[1]{\label{#1}}  
\nc{\mcite}[1]{\cite{#1}}  
\nc{\mref}[1]{\ref{#1}}  
\nc{\mbibitem}[1]{\bibitem{#1}} 

\delete{
	\nc{\mlabel}[1]{\label{#1}  
		{\hfill \hspace{1cm}{\bf{{\ }\hfill(#1)}}}}
	\nc{\mcite}[1]{\cite{#1}{{\bf{{\ }(#1)}}}}  
	\nc{\mref}[1]{\ref{#1}{{\bf{{\ }(#1)}}}}  
	\nc{\mbibitem}[1]{\bibitem[\bf #1]{#1}} 
}

 \font\cyrs=wncyr7

\newcommand{\bk}{{\mathbf{k}}}

\nc{\vep}{\varepsilon}
\nc{\bin}[2]{ (_{\stackrel{\scs{#1}}{\scs{#2}}})}  
\nc{\binc}[2]{(\!\! \begin{array}{c} \scs{#1}\\
		\scs{#2} \end{array}\!\!)}  
\nc{\bincc}[2]{  ( {\scs{#1} \atop
		\vspace{-1cm}\scs{#2}} )}  
\nc{\oline}[1]{\overline{#1}}
\nc{\mapm}[1]{\lfloor\!|{#1}|\!\rfloor}
\nc{\bs}{\bar{S}}
\nc{\la}{\longrightarrow}
\nc{\ot}{\otimes}
\nc{\rar}{\rightarrow}
\nc{\lon }{\,\rightarrow\,}
\nc{\dar}{\downarrow}
\nc{\dap}[1]{\downarrow \rlap{$\scriptstyle{#1}$}}
\nc{\defeq}{\stackrel{\rm def}{=}}
\nc{\dis}[1]{\displaystyle{#1}}
\nc{\dotcup}{\ \displaystyle{\bigcup^\bullet}\ }
\nc{\hcm}{\ \hat{,}\ }
\nc{\hts}{\hat{\otimes}}
\nc{\hcirc}{\hat{\circ}}
\nc{\lleft}{[}
\nc{\lright}{]}
\nc{\curlyl}{\left \{ \begin{array}{c} {} \\ {} \end{array}
	\right .  \!\!\!\!\!\!\!}
\nc{\curlyr}{ \!\!\!\!\!\!\!
	\left . \begin{array}{c} {} \\ {} \end{array}
	\right \} }
\nc{\longmid}{\left | \begin{array}{c} {} \\ {} \end{array}
	\right . \!\!\!\!\!\!\!}
\nc{\ora}[1]{\stackrel{#1}{\rar}}
\nc{\ola}[1]{\stackrel{#1}{\la}}
\nc{\scs}[1]{\scriptstyle{#1}} \nc{\mrm}[1]{{\rm #1}}
\nc{\dirlim}{\displaystyle{\lim_{\longrightarrow}}\,}
\nc{\invlim}{\displaystyle{\lim_{\longleftarrow}}\,}
\nc{\dislim}[1]{\displaystyle{\lim_{#1}}} \nc{\colim}{\mrm{colim}}
\nc{\mvp}{\vspace{0.3cm}} \nc{\tk}{^{(k)}} \nc{\tp}{^\prime}
\nc{\ttp}{^{\prime\prime}} \nc{\svp}{\vspace{2cm}}
\nc{\vp}{\vspace{8cm}}
\nc{\modg}[1]{\!<\!\!{#1}\!\!>}
\nc{\intg}[1]{F_C(#1)}
\nc{\lmodg}{\!<\!\!}
\nc{\rmodg}{\!\!>\!}
\nc{\cpi}{\widehat{\Pi}}
\nc{\ssha}{{\mbox{\cyrs X}}} 
\nc{\tsha}{{\mbox{\cyrt X}}}
\nc{\shpr}{\diamond}    
\nc{\labs}{\mid\!}
\nc{\rabs}{\!\mid}

\nc{\RBA}{\mathsf{ERBPLie}}
\nc{\C}{{\mathrm{C}}}
\nc{\ad}{\mrm{ad}}
\nc{\ann}{\mrm{ann}}
\nc{\Aut}{\mrm{Aut}}
\nc{\DA}{{\mathsf{DL}_\lambda}}
\nc{\Alg}{{\mathrm{Lie}}}
\nc{\DO}{{\mathsf{DO}_\lambda}}
\nc{\bim}{\mbox{-}\mathsf{Rep}}
\nc{\md}{\mbox{-}\mathsf{rep}}
\nc{\br}{\mrm{bre}}
\nc{\can}{\mrm{can}}
\nc{\Cont}{\mrm{Cont}}
\nc{\rchar}{\mrm{char}}
\nc{\cok}{\mrm{coker}}
\nc{\de}{\mrm{dep}}
\nc{\dtf}{{R-{\rm tf}}}
\nc{\dtor}{{R-{\rm tor}}}

\nc{\Div}{{\mrm Div}}
\nc{\Diff}{\mrm{DL}}
\nc{\Diffl}{\mathsf{DL}_\lambda}
\nc{\diffo}{{\mathsf{DO}_\lambda}}
\nc{\Dif}{{\mathfrak{Dif}^\lambda}}
\nc{\Difinfty}{{\mathfrak{Dif}^\lambda_\infty}}
\nc{\alg}{\mathsf{Lie}}
\nc{\End}{\mrm{End}}
\nc{\Ext}{\mrm{Ext}}
\nc{\Fil}{\mrm{Fil}}
\nc{\Fr}{\mrm{Fr}}
\nc{\Frob}{\mrm{Frob}}
\nc{\Gal}{\mrm{Gal}}
\nc{\GL}{\mrm{GL}}
\nc{\Hom}{\mrm{Hom}}
\nc{\Hoch}{\mrm{Hoch}}
\nc{\hsr}{\mrm{H}}
\nc{\hpol}{\mrm{HP}}
\nc{\id}{\mrm{id}}
\nc{\im}{\mrm{im}}
\nc{\Id}{\mrm{Id}}
\nc{\ID}{\mrm{ID}}
\nc{\Irr}{\mrm{Irr}}
\nc{\incl}{\mrm{incl}}
\nc{\length}{\mrm{length}}
\nc{\NLSW}{\mrm{NLSW}}
\nc{\Lie}{\mrm{Lie}}
\nc{\mchar}{\rm char}
\nc{\mpart}{\mrm{part}}
\nc{\ql}{{\QQ_\ell}}
\nc{\qp}{{\QQ_p}}
\nc{\rank}{\mrm{rank}}
\nc{\rcot}{\mrm{cot}}
\nc{\rdef}{\mrm{def}}
\nc{\rdiv}{{\rm div}}
\nc{\rtf}{{\rm tf}}
\nc{\rtor}{{\rm tor}}
\nc{\res}{\mrm{res}}
\nc{\Sh}{{\mathrm{Sh}}}
\nc{\SL}{\mrm{SL}}
\nc{\Spec}{\mrm{Spec}}
\nc{\sgn}{{\mathrm{sgn}}}
\nc{\tor}{\mrm{tor}}
\nc{\Tr}{\mrm{Tr}}
\nc{\tr}{\mrm{tr}}
\nc{\wt}{\mrm{wt}}
\nc{\op}{\mrm{op}}
\nc{\rmH}{ {\mathrm{H}}}

\nc{\bfk}{{\bf k}}
\nc{\bfone}{{\bf 1}}
\nc{\bfzero}{{\bf 0}}
\nc{\detail}{\marginpar{\bf More detail}
	\noindent{\bf Need more detail!}
	\svp}
\nc{\gap}{\marginpar{\bf Incomplete}\noindent{\bf Incomplete!!}
	\svp}
\nc{\FMod}{\mathbf{FMod}}
\nc{\Int}{\mathbf{Int}}
\nc{\Mon}{\mathbf{Mon}}
\nc{\remarks}{\noindent{\bf Remarks: }}
\nc{\Rep}{\mathbf{Rep}}
\nc{\Rings}{\mathbf{Rings}}
\nc{\Sets}{\mathbf{Sets}}
\nc{\ob}{\mathsf{Ob}}

\nc{\BA}{{\mathbb A}}   \nc{\CC}{{\mathbb C}}
\nc{\DD}{{\mathbb D}}   \nc{\EE}{{\mathbb E}}
\nc{\FF}{{\mathbb F}}   \nc{\GG}{{\mathbb G}}
\nc{\HH}{{\mathbb H}}   \nc{\LL}{{\mathbb L}}
\nc{\NN}{{\mathbb N}}   \nc{\PP}{{\mathbb P}}
\nc{\QQ}{{\mathbb Q}}   \nc{\RR}{{\mathbb R}}
\nc{\TT}{{\mathbb T}}   \nc{\VV}{{\mathbb V}}
\nc{\ZZ}{{\mathbb Z}}   \nc{\TP}{\widetilde{P}}
\nc{\m}{{\mathbbm m}}

\nc{\cala}{{\mathcal A}}    \nc{\calc}{{\mathcal C}}
\nc{\cald}{\mathcal{D}}     \nc{\cale}{{\mathcal E}}
\nc{\calf}{{\mathcal F}}    \nc{\calg}{{\mathcal G}}
\nc{\calh}{{\mathcal H}}    \nc{\cali}{{\mathcal I}}
\nc{\call}{{\mathcal L}}    \nc{\calm}{{\mathcal M}}
\nc{\caln}{{\mathcal N}}    \nc{\calo}{{\mathcal O}}
\nc{\calp}{{\mathcal P}}    \nc{\calr}{{\mathcal R}}
\nc{\cals}{{\mathcal S}}    \nc{\calt}{{\Omega}}
\nc{\calv}{{\mathcal V}}    \nc{\calw}{{\mathcal W}}
\nc{\calx}{{\mathcal X}}

\nc{\fraka}{{\mathfrak a}}
\nc{\frakb}{\mathfrak{b}}
\nc{\frakg}{{\frak g}}
\nc{\frakl}{{\frak l}}
\nc{\fraks}{{\frak s}}
\nc{\frakB}{{\frak B}}
\nc{\frakm}{{\frak m}}
\nc{\frakM}{{\frak M}}
\nc{\frakp}{{\frak p}}
\nc{\frakW}{{\frak W}}
\nc{\frakX}{{\frak X}}
\nc{\frakS}{{\frak S}}
\nc{\frakA}{{\frak A}}
\nc{\frakx}{{\frakx}}
\nc{\frakC}{{\frak{C}}}
\nc{\frakh}{{\frak h}}

\nc{\lir}[1]{\textcolor{red}{\underline{Li:}#1 }}

\begin{document}

\title [Modified Rota-Baxter operators]{Modified Rota-Baxter operators of non-zero weight on $3$-Lie algebras}

\author{Shuangjian Guo}
\address{(Shuangjian Guo)
 School of Mathematics and Statistics, Guizhou University of Finance and Economics, Guizhou 550025, P. R. of China}
\email{shuangjianguo@126.com}

\author{Yufei Qin}
\address{(Yufei Qin)   School of Mathematical Sciences \\ Key Laboratory of MEA (Ministry of Education)\\ Shanghai Key laboratory of PMMP\\	East China Normal University\\	Shanghai 200241,	P. R. of China}
     \address{Department of Mathematics and Data Science\\ Vrije Universiteit Brussel\\ Pleinlaan 2, 1050 Brussels\\ Belgium }
\email{Yufei.Qin@vub.be (corresponding author)}

\author{Guodong Zhou}
\address{(Guodong Zhou) School of Mathematical Sciences,  Shanghai Key Laboratory of PMMP,  East China Normal University, Shanghai 200241, P. R. of China}
\email{gdzhou@math.ecnu.edu.cn}

\date{\today}

\begin{abstract}
In this paper, we introduce the notion of  modified Rota-Baxter
operators of non-zero weight on $3$-Lie algebras  and provide some examples. Next,  we give various constructions of modified Rota-Baxter
operators of non-zero weight  according to constructions of $3$-Lie algebras.  Furthermore,   we define a cohomology of  modified Rota-Baxter
operators of non-zero weight  on $3$-Lie algebras with coefficients in a suitable representation.   As an application, we study formal deformations of modified Rota-Baxter
operators of non-zero weight  that are generated by the above-defined cohomology. In the final part of the paper, we construct two \(L_\infty[1]\)-algebra structures whose Maurer-Cartan elements correspond to relative and absolute modified Rota-Baxter \(3\)-Lie algebra structures of nonzero weight, respectively. Lastly, we compare our \(L_\infty[1]\)-algebraic approach with the deformation-controlling \(L_\infty[1]\)-algebra for relative Rota-Baxter \(3\)-Lie operators developed by Hou, Sheng, and Zhou.

\end{abstract}

\subjclass[2010]{   
16B70
17B56  
18D99  
}

\keywords{Modified Rota-Baxter operator,    Cohomology, Nijenhuis operator, Formal deformation, Maurer-Cartan element, L$_\infty[1]$-algebra}

\maketitle

\tableofcontents

\allowdisplaybreaks

\section{Introduction}

Rota-Baxter algebras were first introduced by Baxter in his study of the fluctuation theory in probability \cite{B60}. Baxter's work was further investigated  by Rota \cite{Ro69} and Cartier \cite{C72}.  {Later,   Guo et al. \cite{G00,Gu00, GK00} established many important results concerning Rota–Baxter algebras.}  Bai etc.\cite{BGS19} introduced the notion of relative Rota-Baxter operators  on $3$-Lie algebras when studying the solutions of the $3$-Lie-Yang-Baxter equation. Recently, Tang etc.\cite{THS21} constructed a Lie $3$-algebra, and its Maurer-Cartan element is precisely the relative Rota-Baxter operator on the $3$-Lie algebra.  Later, Hou etc.\cite{HSZ23} introduced the  relative Rota-Baxter operator of non-zero weight on the $3$-Lie algebra and described the  relative Rota-Baxter operator of non-zero weight by introducing the $3$-post-Lie algebra.

The concept of the modified Rota-Baxter Lie algebras \cite{S83} originated from the modified classical Yang-Baxter equation, which was later applied to the study of non-commutative generalized Lax pairs, affine geometry of Lie groups, and the $\mathcal{O}$-operator, etc. \cite{BGN10,B90}. Recently, the authors studied modified Rota-Baxter Leibniz algebras \cite{MS23}, modified $r$-matrices \cite{JS21} and modified Rota-Baxter pre-Lie algebras \cite{GL25}. Due to the particularity of the definition of 3-Lie algebras, how to reasonably introduce the concept of modified Rota-Baxter operators on 3-Lie algebras is {the first motivation of this paper}.

A few years ago, Bai etc \cite{BGLW13} introduced the notion of  Rota-Baxter $3$-Lie algebra with arbitrary weights, and  shown that they can be derived from Rota-Baxter Lie algebras and pre-Lie algebras and from Rota-Baxter commutative associative algebras with derivations.  Recently, we \cite{GQWZ23} studied the cohomology, abelian extension and deformation theory of the Rota-Baxter $3$-Lie algebra with arbitrary weights.  How to construct modified Rota-Baxter 3-Lie algebras from modified Rota-Baxter Lie algebras and {modified Rota-Baxter pre-Lie algebras from modified Rota-Baxter commutative associative algebras with derivations} is the second motivation for studying this paper.

The construction of $L_\infty$-subalgebras using $V$-data has been studied in \cite{DM22, FZ15, LQYZ24}. Hou etc.\cite{HSZ23} constructed the controlling $L_\infty[1]$-algebra structure for relative Rota-Baxter operators of weight~$\lambda$ on $3$-Lie algebras. In the final part of the paper, by using the derived bracket technique,  we construct two \(L_\infty[1]\)-algebra structures whose Maurer-Cartan elements correspond to relative and absolute modified Rota-Baxter \(3\)-Lie algebra structures of nonzero weight, respectively. Lastly, we compare our \(L_\infty[1]\)-algebraic approach with the deformation-controlling \(L_\infty[1]\)-algebra for relative Rota--Baxter \(3\)-Lie operators developed by Hou, Sheng, and Zhou.

The paper is organized as follows. In Section~\ref{Preliminaries}, we recall some basic definitions about  $3$-Lie algebras and their cohomology. In Section~\ref{Modified Rota-Baxter operators of weight}, we introduce the notion of  modified Rota-Baxter
operators of non-zero weight on $3$-Lie algebras  and provide some examples. In Section~\ref{The constructions of modified Rota-Baxter operators},  we give various constructions of modified Rota-Baxter
operators of non-zero weight  according to constructions of $3$-Lie algebras.  In Section~\ref{Cohomology of modified Rota-Baxter operators of weight},   we define a cohomology of  modified Rota-Baxter
operators of non-zero weight  on $3$-Lie algebras with coefficients in a suitable representation. In Section~\ref{Formal deformations  of modified Rota-Baxter operators of weight },  as an application, we study formal deformations of modified Rota-Baxter
operators of non-zero weight  that are generated by the above-defined cohomology. In Section~\ref{L infinty-structure for (relative and absolute) modified Rota-Baxter 3-Lie algebras}, we construct an \( L_\infty[1] \)-algebra structure on the cochain complex associated with (relative and absolute) modified Rota-Baxter \( 3 \)-Lie algebras. We show that the modified Rota-Baxter \( 3 \)-Lie algebra structures of weight $\lambda$ correspond precisely to the Maurer-Cartan elements of this \( L_\infty[1] \)-algebra.  Furthermore, we identify an \( L_\infty[1] \)-subalgebra that controls deformations of relative Rota-Baxter \( 3 \)-Lie algebras of weight \( \lambda \). Under the twisting procedure, this subalgebra yields the \( L_\infty[1] \)-algebra constructed by Hou, Sheng, and Zhou, which controls the deformations of relative Rota-Baxter operators of weight \( \lambda \) on \( 3 \)-Lie algebras.

Throughout this paper, $\bk$ denotes a field of characteristic zero. All the vector spaces,
algebras, linear maps and tensor products are taken over $\bk$ unless otherwise specified. $\lambda\in \bk$ is a fixed nonzero scalar, referred to as the weight.

\smallskip

\section{Preliminaries}\label{Preliminaries}

In this section,  we will recall some basic notions and facts about $3$-Lie algebras from \cite{ BGLW13, Fil85, HSZ23, L16}.

\begin{defn}
	A \textbf{$3$-Lie algebra} is a vector space $\frak g$ together with a skew-symmetric trilinear map
	\[
	[\cdot,\cdot,\cdot]:\wedge^{3}\frak g\to\frak g
	\]
	satisfying the fundamental identity
	\[
	[x,y,[u,v,w]]
	=[[x,y,u],v,w]+[u,[x,y,v],w]+[u,v,[x,y,w]],
	\qquad x,y,u,v,w\in\frak g.
	\]
\end{defn}

\begin{defn}
	Let $(\frak g,[\cdot,\cdot,\cdot])$ be a $3$-Lie algebra and $V$ a vector space.
	A \textbf{representation of $\frak g$ on $V$} is a linear map
	\[
	\rho:\wedge^{2}\frak g\to\mathfrak{gl}(V)
	\]
	such that for all $x_1,x_2,x_3,x_4\in\frak g$,
	\begin{align}
		[\rho(x_1,x_2),\rho(x_3,x_4)]
		&=\rho([x_1,x_2,x_3],x_4)+\rho(x_3,[x_1,x_2,x_4]), \label{Eq: 3-Lie representation1}\\
		\rho(x_1,[x_2,x_3,x_4])
		&=\rho(x_3,x_4)\rho(x_1,x_2)-\rho(x_2,x_4)\rho(x_1,x_3)
		+\rho(x_2,x_3)\rho(x_1,x_4). \label{Eq: 3-Lie representation2}
	\end{align}
\end{defn}

\begin{defn}
	 Let $(\frak g,[\cdot,\cdot,\cdot])$ be a $3$-Lie algebra.
	 Define
	 \[
	 \ad:\wedge^{2}\frak g\to\mathfrak{gl}(\frak g),
	 \qquad
	 \ad_{x,y}(z)=[x,y,z].
	 \]
	 Then $\ad$ is a representation of $\frak g$ on $\frak g$, called the \textbf{adjoint representation}.
\end{defn}

\begin{defn}
	Let $\frak g$ be a $3$-Lie algebra. The \textbf{derived algebra} is
	\[
	\frak g^{1}=[\frak g,\frak g,\frak g],
	\]
	and the center is
	\[
	\mathcal C(\frak g)=\{x\in\frak g\mid [x,y,z]=0,\ \forall y,z\in\frak g\}.
	\]
\end{defn}

\begin{defn}
	Let $(\frak g,[\cdot,\cdot,\cdot])$ and $(\frak h,\{\cdot,\cdot,\cdot\})$ be $3$-Lie algebras, and
	$\rho:\wedge^{2}\frak g\to\mathfrak{gl}(\frak h)$ a representation of $\frak g$ on $\frak h$.
	If
	\[
	\rho(x,y)u\in\mathcal C(\frak h),\qquad
	\rho(x,y)\{u,v,w\}=0,
	\quad
	\forall x,y\in\frak g,\ u,v,w\in\frak h,
	\]
	then $\rho$ is called an \textbf{action of $\frak g$ on $\frak h$}.
\end{defn}

\smallskip
Let $(V,\rho)$ be a representation of a $3$-Lie algebra $\frak g$.
Set
\[
C^0_{\mathrm{3Lie}}(\frak g,V)=V,\qquad
C^n_{\mathrm{3Lie}}(\frak g,V)=
\Hom\big((\wedge^{2}\frak g)^{\otimes(n-1)}\wedge\frak g,\;V\big),\quad n\ge1.
\]
The coboundary operator
\[
\partial_{\mathrm{3Lie}}^n:C^n_{\mathrm{3Lie}}(\frak g,V)\to C^{n+1}_{\mathrm{3Lie}}(\frak g,V)
\]
is given, for $\mathfrak X_i=x_i\wedge y_i$, by
\[
\begin{split}
	\partial_{\mathrm{3Lie}}^n f(\mathfrak X_1,\ldots,\mathfrak X_n,x_{n+1})
	&=(-1)^{n+1}\rho(y_n,x_{n+1})f(\mathfrak X_1,\ldots,\mathfrak X_{n-1},x_n)\\
	&\quad+(-1)^{n+1}\rho(x_{n+1},x_n)f(\mathfrak X_1,\ldots,\mathfrak X_{n-1},y_n)\\
	&\quad+\sum_{j=1}^n(-1)^{j+1}\rho(x_j,y_j)
	f(\mathfrak X_1,\ldots,\widehat{\mathfrak X_j},\ldots,\mathfrak X_n,x_{n+1})\\
	&\quad+\sum_{j=1}^n(-1)^j
	f(\mathfrak X_1,\ldots,\widehat{\mathfrak X_j},\ldots,\mathfrak X_n,[x_j,y_j,x_{n+1}])\\
	&\quad+\sum_{1\le j<k\le n}(-1)^j
	f(\mathfrak X_1,\ldots,\widehat{\mathfrak X_j},\ldots,\mathfrak X_{k-1},\\
	&\hspace{2.8cm}
	[x_j,y_j,x_k]\wedge y_k+x_k\wedge[x_j,y_j,y_k],\mathfrak X_{k+1},\ldots,\mathfrak X_n,x_{n+1}).
\end{split}
\]

The corresponding cohomology is denoted by
\[
H^*_{\mathrm{3Lie}}(\frak g,V).
\]
When $V=\frak g$ with the adjoint representation, we write
\[
H^n_{\mathrm{3Lie}}(\frak g)=H^n_{\mathrm{3Lie}}(\frak g,\frak g),\qquad n\ge1.
\]

Next, we study decompositions of $3$-Lie algebras and their associated product structures.

	\begin{defn}\cite{L16}
		Let $(\frakg,[\cdot,\cdot,\cdot])$ be a $3$-Lie algebra and
		$N:\frakg\to\frakg$ a linear map.
			$N$ is called a \textbf{Nijenhuis operator} on $\frakg$ if
			\begin{align*}
				[N(x),N(y),N(z)]
				&=N\Big(
				[N(x),N(y),z]+[x,N(y),T(z)]+[N(x),y,N(z)] \\
				&\quad -N[N(x),y,z]-N[x,N(y),z]-N[x,y,T(z)]
				\Big).
			\end{align*}
	\end{defn}
	\begin{defn}\cite{ST18}
		Let $(\mathfrak{g},[\cdot,\cdot,\cdot]_{\mathfrak g})$ be a $3$-Lie algebra.
		An \textbf{almost product structure} on $\mathfrak g$ is a linear endomorphism
		$E:\mathfrak g\to\mathfrak g$ such that
		\[
		E^{2}=\mathrm{Id}, \qquad E\neq \pm \mathrm{Id}.
		\]
		An almost product structure $E$ is called a \textbf{product structure} if it
		satisfies the following integrability condition:
		\begin{align*}
			E[x,y,z]_{\mathfrak g}
			={}&[Ex,Ey,Ez]_{\mathfrak g}
			+[Ex,y,z]_{\mathfrak g}
			+[x,Ey,z]_{\mathfrak g}
			+[x,y,Ez]_{\mathfrak g} \\
			&-E[Ex,Ey,z]_{\mathfrak g}
			-E[x,Ey,Ez]_{\mathfrak g}
			-E[Ex,y,Ez]_{\mathfrak g}.
		\end{align*}
	\end{defn}
	\begin{remark}\label{Remark: Nijenhuis operator and product structure}\cite{ST18}
		A product structure on a $3$-Lie algebra can equivalently be regarded as a
		Nijenhuis operator $E$ satisfying $E^{2}=\mathrm{Id}$.
	\end{remark}
\begin{theorem}\label{Thm: product structure and decomposition}[{\cite[Theorem~5.3]{ST18}}]
	Let $\bigl(\mathfrak{g},[\cdot,\cdot,\cdot]_{\mathfrak g}\bigr)$ be a $3$-Lie algebra.
	Then $\mathfrak g$ admits a product structure $E$ if and only if it decomposes as
	\[
	\mathfrak g=\mathfrak g_{+}\oplus \mathfrak g_{-},
	\]
	where $\mathfrak g_{+}$ and $\mathfrak g_{-}$ are $3$-Lie subalgebras of $\mathfrak g$.
	In this case, $E$ is given by
	\[
	E(x,u)=(x,-u), \qquad \forall\, x\in\mathfrak g_{+},\; u\in\mathfrak g_{-}.
	\]
\end{theorem}

\bigskip
 \section{Modified Rota-Baxter operators of weight $\lambda$}\label{Modified Rota-Baxter operators of weight}
 \def\theequation{\arabic{section}.\arabic{equation}}
\setcounter{equation} {0}

 In this section,  we introduce the notions of  Rota-Baxter alegbras of weight $\lambda$ and relative (modified) Rota-Baxter alegbras of weight $\lambda$  give some properties.

\begin{defn}\cite{BGLW13}
		Let $(\frakg,[\cdot,\cdot,\cdot])$ be a $3$-Lie algebra and
		$T:\frakg\to\frakg$ a linear map.
			$T$ is called a \textbf{Rota-Baxter operator of weight $\lambda$} on $\frakg$ if
			\begin{align*}
				[T(x),T(y),T(z)]
				&=T\Big(
				[T(x),T(y),z]+[x,T(y),T(z)]+[T(x),y,T(z)]  \\
				&\quad +\lambda [T(x),y,z]+\lambda[x,T(y),z]+\lambda[x,y,T(z)]
				+\lambda^{2}[x,y,z]
				\Big).
			\end{align*}
\end{defn}

\begin{defn}Let $\left(\mathfrak{g},[\cdot, \cdot, \cdot]\right)$ and $\left(\mathfrak{h},\{\cdot, \cdot, \cdot\}\right)$ be $3$-Lie algebras.
	Let $\rho: \wedge^2 \mathfrak{g} \rightarrow \mathfrak{g l}(\mathfrak{h})$  be an action of a $3$-Lie algebra $\left(\mathfrak{g},[\cdot, \cdot, \cdot]\right)$ on a $3$-Lie algebra $\left(\mathfrak{h},\{\cdot, \cdot, \cdot\}\right)$,  and $\zeta:\wedge^2\mathfrak{h}\rightarrow \mathfrak{g l}(\mathfrak{g})$ an action of   $\left(\mathfrak{h},\{\cdot, \cdot, \cdot\}\right)$ on  $\left(\mathfrak{g},[\cdot, \cdot, \cdot]\right)$. A linear map $R: \mathfrak{h} \rightarrow \mathfrak{g}$ is called a \textbf{relative modified Rota-Baxter  operator of weight $\lambda \in \bk$} from a $3$-Lie algebra $\mathfrak{h}$ to a $3$-Lie algebra $\mathfrak{g}$ with respect to   actions $\rho$ and $\zeta$ if $ \forall u, v, w \in \mathfrak{h}$,
	\begin{align*}
		[R(u), R(v), R(w)]=&R\Big(\rho(R(u), R(v)) w+\rho(R(v), R (w)) u+\rho(R(w), R(u)) v+\lambda \{u, v, w\}\Big)\\
		&-\lambda \zeta(u,v)R(w)-\lambda \zeta(v,(w)R(u)-\lambda \zeta(w,u)R(v).
	\end{align*}
In this case, the quadruple $\left(  \left(\mathfrak{h},\{\cdot, \cdot, \cdot\}\right),\left(\mathfrak{g},[\cdot, \cdot, \cdot]\right),\rho,\zeta \right) $ is called a \textbf{relative $3$-Lie algebra pair}.
Moreover, the quintuple $\left( \left(\mathfrak{h},\{\cdot, \cdot, \cdot\}\right),\left(\mathfrak{g},[\cdot, \cdot, \cdot]\right),\rho,\zeta,R\right) $ is called a \textbf{relative modified Rota-Baxter $3$-Lie algebra of weight $\lambda$}.
	
	Moreover, If
	\begin{itemize}
		\item  [(a)]  $\mathfrak{h}=\mathfrak{g}$ and $\rho=\zeta=\mathrm{ad}$, we call
		$(\mathfrak{g},[\cdot,\cdot,\cdot],R)$ an \textbf{absolute modified Rota-Baxter $3$-Lie algebra of weight $\lambda$},
		or simply a modified Rota-Baxter $3$-Lie algebra of weight $\lambda$,
		or a modified Rota-Baxter $3$-Lie algebra;
		\item [(b)] $\zeta:\wedge^2\mathfrak{h}\rightarrow \mathfrak{g l}(\mathfrak{g})$ vanish, we call the quadruple $\left( \left(\mathfrak{h},\{\cdot, \cdot, \cdot\}\right),\left(\mathfrak{g},[\cdot, \cdot, \cdot]\right),\rho,R\right) $ is called a \textbf{relative  Rota-Baxter $3$-Lie algebra of weight $\lambda$}.
	\end{itemize}
\end{defn}

\begin{exam}
Let $(\mathfrak{g}, [\cdot, \cdot, \cdot])$ be a $3$-Lie algebra. Then $\mathrm{id}_\mathfrak{g}$ is a modified Rota-Baxter operator of weight $1$.
\end{exam}

\begin{exam}
Let $(\mathfrak{g}, [\cdot, \cdot, \cdot])$ be a $3$-Lie algebra.  Then $R$ is a modified Rota-Baxter operator if and only if   $-R$ is a modified Rota-Baxter operator.
\end{exam}

\begin{exam}
 Given a modified Rota-Baxter operator  $R$ and an automorphism $\psi\in \mathrm{Aut}(\frakg)$ of the $3$-Lie
algebra $\frakg$, then  $\psi^{-1} \circ R \circ \psi$  is a modified Rota-Baxter operator.
\end{exam}
\begin{exam}
Let $(\mathfrak{g}, [\cdot, \cdot, \cdot])$ be a $3$-Lie algebra whose non-zero brackets are given with
 respect to a basis $\{e_1, e_2, e_3\}$ by
\begin{align*}
[e_1, e_2, e_3]=e_1.
\end{align*}
Then $R=\left(
                                            \begin{array}{ccc}
                                             a_{11}  & a_{12} & a_{13}\\
                                             a_{21}& a_{22}& a_{23}\\
                                             a_{31}& a_{32}& a_{33}\\
                                            \end{array}
                                          \right)$ is a modified Rota-Baxter operator of weight $1$ if and only if
\begin{small}
\begin{align*}
[R(e_1), R(e_2), R(e_3)] &=R( [R(e_1), R(e_2), e_3]+[e_1, R(e_2), R(e_3)]+[R(e_1), e_2, R(e_3)]+[e_1, e_2, e_3])\\
\quad &\quad- [R(e_1), e_2, e_3]-[e_1, R(e_2), e_3]-[e_1, e_2, R(e_3)].
\end{align*}
\end{small}
After calculation, we can obtain
\begin{align*}
[R(e_1),R(e_2),R(e_3)]&=[a_{11}e_1, a_{22}e_2, a_{33}e_3]+[a_{11}e_1, a_{32}e_3, a_{23}e_2]\\
\quad &\quad+[a_{21}e_2, a_{12}e_1, a_{33}e_3]+[a_{21}e_2, a_{32}e_3, a_{13}e_1]\\
\quad\quad &\quad+[a_{31}e_3, a_{12}e_1, a_{23}e_2]+[a_{31}e_3, a_{22}e_2, a_{13}e_1]\\
&= (a_{11}a_{22}a_{33}-a_{11}a_{32}a_{23}-a_{21}a_{12}a_{33} + a_{21}a_{32}a_{13} + a_{31}a_{12}a_{23}-a_{31}a_{22}a_{13})e_1,
\end{align*}
on the other hand,
\begin{align*}
&R( [R(e_1), R(e_2), e_3]+[e_1, R(e_2), R(e_3)]+[R(e_1), e_2, R(e_3)]+[e_1, e_2, e_3])\\
\quad\quad &\quad- [R(e_1), e_2, e_3]-[e_1, R(e_2), e_3]-[e_1, e_2, R(e_3)]\\
&= R([a_{11}e_1, a_{22}e_2, e_3]) + R([a_{21}e_2, a_{12}e_1, e_3]) + R([a_{11}e_1, e_2, a_{33}e_3]) + R([a_{31}e_3, e_2, a_{13}e_1])\\
\quad\quad &\quad+R([e_1, a_{22}e_2, a_{33}e_3]) + R([e_1, a_{32}e_3, a_{23}e_2])+R([e_1, e_2, e_3])\\
\quad\quad &\quad- [a_{11}e_1, e_2, e_3]-[e_1, a_{22}e_2, e_3]-[e_1, e_2, a_{33}e_3]\\
&=\big((a_{11}a_{22}-a_{21}a_{12} + a_{11}a_{33}-a_{31}a_{13} + a_{22}a_{33}-a_{32}a_{23})a_{11}-a_{22}-a_{33}\big)e_1\\
\quad\quad &\quad+(a_{11}a_{22}-a_{21}a_{12} + a_{11}a_{33}-a_{31}a_{13} + a_{22}a_{33}-a_{32}a_{23}+1)a_{21}e_2\\
\quad\quad &\quad+(a_{11}a_{22}-a_{21}a_{12} + a_{11}a_{33}-a_{31}a_{13} + a_{22}a_{33}-a_{32}a_{23}+1)a_{31}e_3.
\end{align*}
Thus, $R$ is a modified Rota-Baxter operator of weight $1$ if and only if
\begin{align*}
&a_{11}a_{22}a_{33}-a_{11}a_{32}a_{23}-a_{21}a_{12}a_{33} + a_{21}a_{32}a_{13} + a_{31}a_{12}a_{23}-a_{31}a_{22}a_{13}\\
&=(a_{11}a_{22}-a_{21}a_{12} + a_{11}a_{33}-a_{31}a_{13} + a_{22}a_{33}-a_{32}a_{23})a_{11}-a_{22}-a_{33},
\end{align*}
and
\begin{align*}
&(a_{11}a_{22}-a_{21}a_{12} + a_{11}a_{33}-a_{31}a_{13} + a_{22}a_{33}-a_{32}a_{23}+1)a_{21}\\
&= (a_{11}a_{22}-a_{21}a_{12} + a_{11}a_{33}-a_{31}a_{13} + a_{22}a_{33}-a_{32}a_{23}+1)a_{31}\\
&=0.
\end{align*}
In particular, $R_1=\left(
                                            \begin{array}{ccc}
                                             1  & 0 & 0\\
                                             0& 1& 0\\
                                             0& 0& -1\\
                                            \end{array}
                                          \right)$  \text { and }   $R_2=\left(
                                         \begin{array}{ccc}
                                             1  & 0 & 0\\
                                             0& -1& 0\\
                                             0& 0& 1\\
                                            \end{array}
                                          \right)$  are modified Rota-Baxter operators of weight $1$.

\end{exam}
\begin{exam}
Let $(\mathfrak{g}, [\cdot, \cdot, \cdot])$ be a $3$-Lie algebra whose non-zero brackets are given with
 respect to a basis $\{e_1, e_2, e_3, e_4\}$ by
\begin{align*}
[e_2, e_3, e_4]=e_1.
\end{align*}
Then $R=\left(
                                            \begin{array}{cccc}
                                             a_{11}  & a_{12} & a_{13}& a_{14}\\
                                             a_{21}& a_{22}& a_{23}& a_{24}\\
                                             a_{31}& a_{32}& a_{33}& a_{34}\\
                                             a_{41}& a_{42}& a_{43}& a_{44}\\
                                            \end{array}
                                          \right)$ is a modified Rota-Baxter operator of weight $1$ if and only if
\begin{small}
\begin{align*}
[R(e_2), R(e_3), R(e_4)] &=R( [R(e_2), R(e_3), e_4]+[e_2, R(e_3), R(e_4)]+[R(e_2), e_3, R(e_4)]+[e_2, e_3, e_4])\\
\quad &\quad- [R(e_2), e_3, e_4]-[e_2, R(e_3), e_4]-[e_2, e_3, R(e_4)].
\end{align*}
\end{small}
After calculation, we can obtain
\begin{align*}
[R(e_2), R(e_3), R(e_4)]&=[a_{22}e_2, a_{33} e_3, a_{44}e_4]+[a_{22}e_2, a_{43} e_4, a_{34}e_3]+[a_{32}e_3, a_{23} e_2, a_{44}e_4]\\
&\quad + [a_{32}e_3, a_{43} e_4, a_{24}e_2]+[a_{42}e_4, a_{23} e_2, a_{34}e_3]+[a_{42}e_4, a_{33} e_3, a_{24}e_2]\\
&=(a_{22}a_{33}a_{44}-a_{22}a_{43}a_{34}-a_{32}a_{23}a_{44}+a_{32}a_{43}a_{24}+a_{42}a_{23}a_{34}-a_{42}a_{33}a_{24})e_1,
\end{align*}
on the other hand,
\begin{align*}
&R( [R(e_2), R(e_3), e_4]+[e_2, R(e_3), R(e_4)]+[R(e_2), e_3, R(e_4)]+[e_2, e_3, e_4])\\
\quad &\quad- [R(e_2), e_3, e_4]-[e_2, R(e_3), e_4]-[e_2, e_3, R(e_4)]\\
&=R([a_{22}e_2, a_{33}e_3, e_4]) +R([a_{32}e_3,a_{23}e_2, e_4])+ R([e_2, a_{33}e_3, a_{44}e_4])\\
&\quad+R([e_2, a_{43}e_4, a_{34}e_3]) +R([a_{22}e_2, e_3, a_{44}e_4])+R([a_{42}e_4, e_3, a_{24}e_2])\\
&\quad +R(e_1)-a_{22}e_1-a_{33}e_1-a_{44}e_1\\
&=(a_{22}a_{33}-a_{32}a_{23}+a_{33}a_{44}-a_{43}a_{34}+a_{22}a_{44}-a_{42}a_{24}+1)R(e_1)-a_{22}e_1-a_{33}e_1-a_{44}e_1\\
&=((a_{22}a_{33}-a_{32}a_{23}+a_{33}a_{44}-a_{43}a_{34}+a_{22}a_{44}-a_{42}a_{24}+1)a_{11}-a_{22}-a_{33}-a_{44})e_1\\
&\quad +(a_{22}a_{33}-a_{32}a_{23}+a_{33}a_{44}-a_{43}a_{34}+a_{22}a_{44}-a_{42}a_{24}+1)a_{21}e_2\\
&\quad+ (a_{22}a_{33}-a_{32}a_{23}+a_{33}a_{44}-a_{43}a_{34}+a_{22}a_{44}-a_{42}a_{24}+1)a_{31}e_3\\
&\quad+ (a_{22}a_{33}-a_{32}a_{23}+a_{33}a_{44}-a_{43}a_{34}+a_{22}a_{44}-a_{42}a_{24}+1)a_{41}e_4.
\end{align*}
Thus, $R$ is a modified Rota-Baxter operator of weight $1$ if and only if
\begin{align*}
&a_{22}a_{33}a_{44}-a_{22}a_{43}a_{34}-a_{32}a_{23}a_{44}+a_{32}a_{43}a_{24}+a_{42}a_{23}a_{34}-a_{42}a_{33}a_{24}\\
&=(a_{22}a_{33}-a_{32}a_{23}+a_{33}a_{44}-a_{43}a_{34}+a_{22}a_{44}-a_{42}a_{24}+1)a_{11}-a_{22}-a_{33}-a_{44},
\end{align*}
and
\begin{align*}
&(a_{22}a_{33}-a_{32}a_{23}+a_{33}a_{44}-a_{43}a_{34}+a_{22}a_{44}-a_{42}a_{24}+1)a_{21}\\
&=(a_{22}a_{33}-a_{32}a_{23}+a_{33}a_{44}-a_{43}a_{34}+a_{22}a_{44}-a_{42}a_{24}+1)a_{31}\\
&=(a_{22}a_{33}-a_{32}a_{23}+a_{33}a_{44}-a_{43}a_{34}+a_{22}a_{44}-a_{42}a_{24}+1)a_{41}=0.
\end{align*}
In particular, $R_1=\left(
                                            \begin{array}{cccc}
                                             1  & 0 & 0 & 0\\
                                             0& 1& 0& 0\\
                                             0& 0& -1 & 0\\
                                               0& 0& 0 & -1\\
                                            \end{array}
                                          \right)$ \text { and }   $R_2=\left(
                                            \begin{array}{cccc}
                                             1  & 0 & 0 & 0\\
                                             0& -1& 0& 0\\
                                             0& 0& 1 & 0\\
                                               0& 0& 0 & -1\\
                                            \end{array}
                                          \right)$ are modified Rota-Baxter operators of weight $1$.

\end{exam}
\begin{remark}
	Let $R:\mathfrak{g}\to\mathfrak{g}$ be a linear map satisfying $R^{2}=\mathrm{id}$.
	Then the following conditions are equivalent:
	\begin{itemize}
		\item[(a)] $R$ is a modified Rota--Baxter operator of weight $1$;
		
		\item[(b)] $R$ is a Nijenhuis operator;
		
		\item[(c)] $R$ is a product structure;
		
		\item[(d)] There exists a direct sum decomposition of vector spaces
		\[
		\mathfrak{g}=\mathfrak{g}_{1}\oplus \mathfrak{g}_{2},
		\]
		where $\mathfrak{g}_{1}$ and $\mathfrak{g}_{2}$ are subalgebras of $\mathfrak{g}$.
	\end{itemize}
\end{remark}

\begin{proof}
	Assume that $R$ is a modified Rota-Baxter operator of weight $1$.
	A direct computation shows that $R$ satisfies the integrability condition of a product structure.
	Since $R^{2}=\mathrm{id}$, it follows that $R$ defines a product structure.
	Conversely, every product structure with $R^{2}=\mathrm{id}$ is a modified Rota-Baxter operator of weight $1$.
	Hence, $(a)$ and $(c)$ are equivalent.
	
	The equivalence between $(b)$ and $(c)$ follows from Remark~\ref{Remark: Nijenhuis operator and product structure},
	and the equivalence between $(c)$ and $(d)$ follows from Theorem~\ref{Thm: product structure and decomposition}.
	Therefore, the four conditions $(a)$--$(d)$ are equivalent.
\end{proof}

\begin{prop}
Let $(\mathfrak{g}, [\cdot, \cdot, \cdot])$ be a $3$-Lie algebra.  Then $R$ is a  Rota-Baxter operator of weight $\lambda$ if and only if  $ 2R+ \lambda~\mathrm{id}$ is a modified Rota-Baxter operator of weight $\lambda^2$.
\end{prop}
\begin{proof}
For any $x, y, z\in \mathfrak{g}$, we have
\begin{small}
\begin{align*}
& [(2R+\lambda~\mathrm{id})(x),  (2R+\lambda~\mathrm{id})(y), (2R+ \lambda~\mathrm{id})(z)]\\
\quad &=[2Rx+ \lambda x,  2Ry+ \lambda y, 2Rz+ \lambda z]\\
\quad &= 8[Rx, Ry, Rz]+4\lambda[Rx, Ry, z]+4\lambda[R(x), y, Rz]+4\lambda[x, Ry, Rz]\\
\quad &\quad+2\lambda^2[x, y, Rz]+2\lambda^2[x, Ry, z]+2\lambda^2[Rx, y, z]+\lambda^3[x, y, z]\\
\quad &= 8R\Big([Rx, Ry, z] + [Rx, y, Rz] + [x, Ry, Rz]+\lambda[Rx, y, z]+\lambda[x, Ry, z]+\lambda[x, y, Rz]+\lambda^2[x, y, z]\Big)\\
\quad &\quad+4\lambda[Rx, Ry, z]+4\lambda[R(x), y, Rz]+4\lambda[x, Ry, Rz]\\
\quad &\quad+2\lambda^2[x, y, Rz]+2\lambda^2[x, Ry, z]+2\lambda^2[Rx, y, z]+\lambda^3[x, y, z]\\
\quad &= (2R+\lambda~ \mathrm{id})\Big( [(2R+\lambda ~\mathrm{id})x, (2R+\lambda~ \mathrm{id})y, z]+[x, (2R+\lambda ~\mathrm{id})y, (2R+\lambda ~\mathrm{id})z]\\
\quad &\quad+[(2R+\lambda~\mathrm{id})x, y, (2R+ \lambda~\mathrm{id})z]+\lambda^2[x, y, z]\Big)-\lambda^2 [(2R+\lambda~\mathrm{id})x, y, z]\\
&\quad-\lambda^2[x, (2R+\lambda~\mathrm{id})y, z]-\lambda^2[x, y, (2R+\lambda ~\mathrm{id})z].
\end{align*}
\end{small}
And the proof is finished.
\end{proof}

 Let $(\mathfrak{g}, [\cdot, \cdot, \cdot])$ be a $3$-Lie algebra and  $R$ be a modified  Rota-Baxter operator of weight $\lambda$,  by \cite{HSZ23}, in order to make $(\mathfrak{g}, [\cdot, \cdot, \cdot]_R)$ a $3$-Lie algebra, we always assume $\mathfrak{g}^1\subset \mathcal{C}(\mathfrak{g})$, where
\begin{align}
[x, y, z]_R=[Rx, Ry, z]+[x, Ry, Rz]+[Rx, y, Rz]+\lambda[x, y, z],\ \ \forall x, y, z\in \mathfrak{g}. \label{3.2}
\end{align}

\begin{prop}
Let $R$ be a  modified  Rota-Baxter operator of weight $\lambda$ on a $3$-Lie algebra $(\mathfrak{g}, [\cdot, \cdot, \cdot])$.  Then $R$ is  a  modified  Rota-Baxter operator of weight $\lambda$ on a 3-Lie algebra $(\mathfrak{g}, [\cdot, \cdot, \cdot]_R)$.
\end{prop}
\begin{proof} For any  $x, y, z\in \mathfrak{g}$, we have
\begin{small}
\begin{align*}
&[R(x), R(y), R(z)]_R\\
&=[R^2(x), R^2(y), R(z)]+[R(x), R^2(y), R^2(z)]+[R^2(x), R(y), R^2(z)]+\lambda[R(x), R(y), R(z)]\\
&= R( [R^2(x), R^2(y), z]+[R(x), R^2(y), R(z)]+[R^2(x), R(y), R(z)]+\lambda[R(x), R(y), z])- \lambda[R^2(x), R(y), z]\\
\quad&\quad-\lambda[R(x), R^2(y), z]-\lambda[Rx, R(y), R(z)]+R( [R(x), R^2(y), R(z)]+[x, R^2(y), R^2z]+[R(x), R(y), R^2(z)]\\
\quad&\quad+\lambda[x, R(y), R(z)])- \lambda[R(x), R(y), R(z)]-\lambda[x, R^2(y), R(z)]-\lambda[x, R(y), R^2(z)]+R( [R^2(x), R(y), R(z)]\\
\quad&\quad+[R(x), R(y), R^2(z)]+[R^2(x), y, R^2(z)]+\lambda[R(x), y, R(z)])- \lambda[R^2(x), y, R(z)]-\lambda[R(x), R(y), R(z)]\\
\quad&\quad-\lambda[R(x), y, R^2(z)]+\lambda[R(x), R(y), R(z)]\\
&=R( [R^2(x), R^2(y), z]+[R(x), R^2(y), R(z)]+[R^2(x), R(y), R(z)]+\lambda[R(x), R(y), z])- \lambda[R^2(x), R(y), z]\\
\quad&\quad-\lambda[R(x), R^2(y), z]+R( [R(x), R^2(y), R(z)]+[x, R^2(y), R^2z]+[R(x), R(y), R^2(z)]+\lambda[x, R(y), R(z)])\\
\quad&\quad-\lambda [R(x), R(y), R(z)]-\lambda[x, R^2(y), R(z)]-\lambda[x, R(y), R^2(z)]+R( [R^2(x), R(y), R(z)]+[R(x), R(y), R^2(z)]\\
\quad&\quad+[R^2(x), y, R^2(z)]+\lambda[R(x), y, R(z)])- \lambda[R^2(x), y, R(z)]-\lambda[R(x), R(y), R(z)]-\lambda[R(x), y, R^2(z)]\\
&=R( [R(x), R(y), z]_R+[x, R(y), R(z)]_R+[R(x), y, R(z)]_R+\lambda[x, y, z]_R)\\
\quad&\quad-\lambda [R(x), y, z]_R-\lambda[x, R(y), z]_R-\lambda[x, y, R(z)]_R.
\end{align*}
\end{small}
Therefore, $R$ is  a  modified  Rota-Baxter operator of weight $\lambda$ on a $3$-Lie algebra $(\mathfrak{g}, [\cdot, \cdot, \cdot]_R)$.
\end{proof}
\medskip

\section{The constructions of modified Rota-Baxter operators of weight  $\lambda$}\label{The constructions of modified Rota-Baxter operators}
\def\theequation{\arabic{section}.\arabic{equation}}
\setcounter{equation} {0}

In this section, according to constructions of $3$-Lie algebras,  we give various constructions of  modified Rota-Baxter operators of weight  $\lambda$.
\begin{lem}(\cite{B10})\label{4.1}
Let $(\mathfrak{g}, [\cdot, \cdot])$ be a Lie algebra and $\mathfrak{g}^{\ast}$ be the dual space of $\mathfrak{g}$.
Suppose that $f\in \mathfrak{g}^{\ast}$ satisfies $f([x, y]) = 0$ for all $x, y \in  \mathfrak{g}$.
Then there is a  $3$-Lie algebra structure on $\mathfrak{g}$ given by
\begin{align*}
[x, y, z]_f = f(x)[y, z] + f(y)[z, x] + f(z)[x, y], \quad \quad \forall x, y, z \in \mathfrak{g}.
\end{align*}
\end{lem}
\begin{defn}
 A linear map $R: \frakg\rightarrow \frakg$ is called a modified Rota-Baxter operator of weight $\lambda\in \mathbf{k}$, or simply a modified Rota-Baxter operator on a Lie algebra $(\frakg, [\cdot, \cdot])$ if $R$ satisfies the following condition
 \begin{align*}
[R(x), R(y)]=R\Big([R(x), y]+[x,R(y)]\Big)+\lambda [x, y], \quad \quad \forall x, y\in \mathfrak{g}.
\end{align*}
\end{defn}
\begin{theorem}\label{4.2}
Let $R$ be a   modified Rota-Baxter operator of weight $\lambda$ on a Lie algebra $(\mathfrak{g}, [\cdot, \cdot])$, $f\in \mathfrak{g}^{\ast}$ satisfies $f([x, y]) = 0$ for any $x, y \in  \mathfrak{g}$.  Then $R$ is  a   modified Rota-Baxter operator of weight  $\lambda$ on a $3$-Lie algebra $(\mathfrak{g}, [\cdot,\cdot,\cdot]_f)$ if and only if  $R$ satisfies
\begin{align*}
&R\Big(f(x)[Ry, Rz]+f(y)[Rz, Rx]+f(z)[Rx, Ry]+\lambda f(x)[y, z] +\lambda f(y)[z, x] + \lambda f(z)[x, y]\Big)\\
&\quad-\lambda f(x)[Ry, z]-\lambda f(x)[y, Rz] - \lambda f(y)[Rz, x]-\lambda  f(y)[z, Rx]-\lambda f(z)[Rx, y]-\lambda f(z)[x, Ry]\\
&\quad-2\lambda f(Ry)[z, x]-2\lambda f(Ry)[z, x]-2\lambda f(Rz)[x, y]=0.
\end{align*}
\end{theorem}
\begin{proof} For any $x, y, z\in \mathfrak{g}$, we have
\begin{align*}
[Rx, Ry, Rz]_f&=f(Rx)[Ry, Rz] + f(Ry)[Rz, Rx] + f(Rz)[Rx, Ry]\\
&=R(f(Rx)[Ry, z]+f(Rx)[y, Rz])+R(f(Ry)[Rz, x]+f(Ry)[z, Rx])\\
\quad&\quad+R(f(Rz)[Rx, y]+f(Rz)[Rx, y])+\lambda f(Rx)[y, z]+\lambda f(Ry)[z, x]\\
 \quad&\quad+\lambda f(Rz)[x, y].
\end{align*}
On the other hand, we have
\begin{align*}
&R( [Rx, Ry, z]_f+[x, Ry, Rz]_f+[Rx, y, Rz]_f+\lambda[x, y, z]_f)- \lambda[Rx, y, z]_f-\lambda[x, Ry, z]_f-\lambda[x, y, Rz]_f\\
\quad&=R\Big( f(Rx)[Ry, z] + f(Ry)[z, Rx] + f(z)[Rx, Ry]+f(x)[Ry, Rz] + f(Ry)[Rz, x] + f(Rz)[x, Ry] \\
\quad&\quad+ f(Rx)[y, Rz] + f(y)[Rz, Rx] + f(Rz)[Rx, y]+\lambda f(x)[y, z] +\lambda f(y)[z, x] + \lambda f(z)[x, y]\Big)\\
&\quad-\lambda f(Rx)[y, z] -\lambda f(y)[z, Rx]-\lambda f(z)[Rx, y]-\lambda f(x)[Ry, z] -\lambda f(Ry)[z, x] - \lambda f(z)[x, Ry]\\
&\quad-\lambda f(x)[y, Rz] -\lambda f(y)[Rz, x] - \lambda f(Rz)[x, y]\\
&=R\Big(f(x)[Ry, Rz]+f(y)[Rz, Rx]+f(z)[Rx, Ry]+\lambda f(x)[y, z] +\lambda f(y)[z, x] + \lambda f(z)[x, y]\Big)\\
&\quad-\lambda f(x)[Ry, z]-\lambda f(x)[y, Rz] - \lambda f(y)[Rz, x]-\lambda  f(y)[z, Rx]-\lambda f(z)[Rx, y]-\lambda f(z)[x, Ry]\\
&\quad +[Rx, Ry, Rz]_f-2\lambda f(Ry)[z, x]-2\lambda f(Rx)[y, z]-2\lambda f(Rz)[x, y].
\end{align*}
 Then $R$ is  a   modified Rota-Baxter operator of weight  $\lambda$ on a $3$-Lie algebra $(\mathfrak{g}, [\cdot,\cdot,\cdot]_f)$ if and only if  $R$ satisfies
\begin{align*}
&R\Big(f(x)[Ry, Rz]+f(y)[Rz, Rx]+f(z)[Rx, Ry]+\lambda f(x)[y, z] +\lambda f(y)[z, x] + \lambda f(z)[x, y]\Big)\\
&\quad-\lambda f(x)[Ry, z]-\lambda f(x)[y, Rz] - \lambda f(y)[Rz, x]-\lambda  f(y)[z, Rx]-\lambda f(z)[Rx, y]-\lambda f(z)[x, Ry]\\
&\quad-2\lambda f(Ry)[z, x]-2\lambda f(Ry)[z, x]-2\lambda f(Rz)[x, y]=0.
\end{align*}
\end{proof}

\begin{exam}
Let $(\mathfrak{g}, [\cdot, \cdot])$ be the 3-dimensional Lie algebra given by
\begin{eqnarray*}
[e_1, e_2]=e_2,
\end{eqnarray*}
where $\{e_1, e_2, e_3\}$ is a basis of $\mathfrak{g}$. By Lemma \ref{4.1}, the trace function $f\in \mathfrak{g}^{\ast}$, where
$\begin{cases}
f(e_1)=1,\\
f(e_2)=0,\\
f(e_3)=1,
\end{cases}$   induces
a 3-Lie algebra $(\mathfrak{g}, [\cdot, \cdot, \cdot]_f)$ defined with the same basis by
\begin{eqnarray*}
[e_1, e_2, e_3]_f=e_2.
\end{eqnarray*}
Consider a linear map $R: \mathfrak{g}\rightarrow \mathfrak{g}$ defined by
$\left(
                                            \begin{array}{ccc}
                                              a_{11} & a_{12} & a_{13}\\
                                             a_{21} & a_{22} & a_{23} \\
                                              a_{31} & a_{32} & a_{33} \\
                                            \end{array}
                                          \right)$ with respect to the basis $\{e_1, e_2, e_3\}$.
   Define
   \begin{eqnarray*}
Re_1 = a_{11}e_1 + a_{21}e_2 + a_{31}e_3,\ \  Re_2 = a_{12}e_1 + a_{22}e_2 + a_{32}e_3, \ \ Re_3 = a_{13}e_1 + a_{23}e_2 + a_{33}e_3.
   \end{eqnarray*}
In order to obtain that $R$ is a modified Rota-Baxter operator of weight $1$ on the Lie algebra $\mathfrak{g}$, we need
\begin{eqnarray*}
&&[Re_i, Re_j] = R([Re_i, e_j] + [e_i, Re_j])+[e_i, e_j],\ \  i, j=1,2,3.
\end{eqnarray*}
By a straightforward computation, we conclude that $R$ is a modified Rota-Baxter operator of weight $1$  on the Lie algebra $\mathfrak{g}$ if and
only if
\begin{eqnarray*}
&&  a_{11}=-a_{22}, \quad  a_{22}^2+a_{12}a_{21}+1=0.
\end{eqnarray*}
By Theorem \ref{4.2},  $R$ is  a modified Rota-Baxter operator of weight $1$ on the 3-Lie algebra $(\mathfrak{g}, [\cdot, \cdot, \cdot]_f)$ if and only if
\begin{align*}
&2a_{33}+2a_{22}a_{32}a_{23} + a_{21}a_{32}a_{13} + a_{31}a_{12}a_{23}-2a_{22}=0,
\end{align*}
and
\begin{align*}
 a_{31}a_{13}+a_{32}a_{23}=2.
 \end{align*}
\end{exam}

\begin{defn} \cite{GL25} A linear map $R: \frakg\rightarrow \frakg$ is called a modified Rota-Baxter operator of weight $\lambda$ on a pre-Lie algebra $(\frakg, \ast)$ if $R$ satisfies the following condition
\begin{align*}
 R(x)\ast R(y) = R(R(x)\ast y + x\ast R(y))+\lambda x\ast y, \ \ \forall x, y \in \mathfrak{g},
\end{align*}
\end{defn}

The following conclusions can be directly verified by definition, so the proof is omitted.

\begin{lem}\label{4.7} Let $R$ be a modified Rota-Baxter operator of weight $\lambda$ on a pre-Lie algebra $(\frakg, \ast)$. Then $R$ is a modified Rota-Baxter operator of weight $\lambda$ on a Lie algebra $(\mathfrak{g},[\cdot, \cdot]_\ast)$, where $[\cdot, \cdot]_\ast$ is defined by
\begin{align*}
[x, y]_\ast=x\ast y-y\ast x, \ \ \forall x, y \in \mathfrak{g}.
\end{align*}
\end{lem}
\begin{lem}Let $R$ be a modified Rota-Baxter operator of weight $\lambda$ on a commutative associative algebra $(\frakg, \cdot)$,  $D$ be a derivation with $D\circ R=R\circ D$. Then  $R$ is a modified Rota-Baxter operator of weight $\lambda$ on a pre-Lie algebra $(\frakg, \ast)$,  where
\begin{align*}
x\ast y=Dx\cdot y, \ \ \forall x, y \in \frakg.
\end{align*}
\end{lem}

According to Theorem \ref{4.2} and Lemma \ref{4.7}, we can know

\begin{theorem} \label{4.8} Let $R$ be a modified Rota-Baxter operator of weight $\lambda$ on a pre-Lie algebra $(\frakg, \ast)$, $f\in \mathfrak{g}^{\ast}$ satisfies $f(x\ast y-y\ast x) = 0$ for any  $x, y \in  \mathfrak{g}$. Define
\begin{align*}
[x, y, z]_f = f(x)(y\ast z-z\ast y)+ f(y)(z\ast x-x\ast z) + f(z)(x\ast y-y\ast x), \forall x, y, z \in \mathfrak{g},
\end{align*}
Then  $R$ is  a   modified Rota-Baxter operator of weight $\lambda$ on a $3$-Lie algebra $(\mathfrak{g}, [\cdot,\cdot,\cdot]_f)$ if and only if  $R$ satisfies
\begin{align*}
&R\Big(f(x)(Ry\ast Rz-Rz\ast Ry)+f(y)(Rz\ast Rx-Rx\ast Rz)+f(z)(Rx\ast Ry-Ry\ast Rx)+\lambda f(x)(y\ast z-z\ast y) \\
&\quad +\lambda f(y)(z\ast x-x\ast z)+ \lambda f(z)(x\ast y-y\ast x)\Big)-\lambda f(x)(Ry\ast z-z\ast Ry)\\
&\quad-\lambda f(x)(y\ast Rz-Rz\ast y) - \lambda f(y)(Rz\ast x-x\ast Rz)-\lambda  f(y)(z\ast Rx-Rx\ast z)\\
&\quad-\lambda f(z)(Rx\ast y-y\ast Rx)-\lambda f(z)(x\ast Ry-Ry\ast x)-2\lambda f(Ry)(z\ast x-x\ast z)\\
&\quad-2\lambda f(Rx)(y\ast z-z\ast y)-2\lambda f(Rz)(x\ast y-y\ast x)=0.
\end{align*}
\end{theorem}

\begin{lem}(\cite{BGLW13})
Let $(\frakg, \cdot)$ be a commutative associative algebra and  $D$ be a derivation and $f\in \frakg^{\ast}$ satisfies $f(D(x)\cdot y)= f( x\cdot D(y))$  for all $x, y \in  \frakg$. Then $(\frakg, \{\cdot, \cdot, \cdot\}_{f, D})$ is a $3$-Lie algebra, where the bracket is given
\begin{eqnarray}
\{x, y, z\}_{f, D}=\left|
                                            \begin{array}{ccc}
                                              f(x) & f(y) & f(z)\\
                                             D(x) & D(y) & D(z) \\
                                              x & y & z \\
                                            \end{array}
                                          \right|,
\end{eqnarray}
for any $x, y, z\in \frakg$.
\end{lem}
\begin{theorem}
 Let $(\frakg, \cdot, R)$ be a commutative modified Rota-Baxter  algebra of weight $\lambda$ and  $D$ be a derivation with $D\circ R=R\circ D$ and $f\in \frakg^{\ast}$ satisfies $f(D(x)\cdot y)= f( x\cdot D(y))$.  Then $R$ is a modified Rota-Baxter operator weight $\lambda$ on the $3$-Lie algebra $(\frakg, \{\cdot, \cdot, \cdot\}_{f, D})$ if and only if $R$ satisfies
\begin{align*}
&R\Big(f(x)(DRy\cdot Rz-DRz\cdot Ry)+f(y)(DRz\cdot Rx-DRx\cdot Rz)+f(z)(DRx\cdot Ry-DRy\cdot Rx)\\
&\quad+\lambda f(x)(Dy\cdot z-Dz\cdot y)  +\lambda f(y)(Dz\cdot x-Dx\cdot z)+ \lambda f(z)(Dx\cdot y-Dy\cdot x)\Big)\\
&\quad-\lambda f(x)(DRy\cdot z-Dz\cdot Ry)-\lambda f(x)(Dy\cdot Rz-DRz\cdot y) - \lambda f(y)(DRz\cdot x-Dx\cdot Rz)\\
&\quad-\lambda  f(y)(Dz\cdot Rx-DRx\cdot z)-\lambda f(z)(DRx\cdot y-Dy\cdot Rx)-\lambda f(z)(Dx\cdot Ry-DRy\cdot x)\\
&\quad-2\lambda f(Ry)(Dz\cdot x-Dx\cdot z)-2\lambda f(Rx)(Dy\cdot z-Dz\cdot y)-2\lambda f(Rz)(Dx\cdot y-Dy\cdot x)=0.
\end{align*}
\end{theorem}
\begin{proof}
The result follows directly from Lemma \ref{4.7} and Theorem \ref{4.8}.
\end{proof}

\bigskip

\section{Cohomology of modified Rota-Baxter operators of weight $\lambda$}\label{Cohomology of modified Rota-Baxter operators of weight}
\def\theequation{\arabic{section}.\arabic{equation}}
\setcounter{equation} {0}
In this section, we construct a representation of the $3$-Lie algebra $(\mathfrak{g}, [\cdot, \cdot, \cdot]_R)$ on the vector space
$\mathfrak{g}$, and define the cohomology of modified Rota-Baxter operators of weight $\lambda$ on $3$-Lie algebras.

\begin{lem}
Let $R: \mathfrak{g}\rightarrow \mathfrak{g} $ be a modified Rota-Baxter operator of weight $\lambda$ on a $3$-Lie algebra $(\mathfrak{g}, [\cdot, \cdot, \cdot])$. Define $\rho_R: \mathfrak{g}\wedge \mathfrak{g}\rightarrow \mathfrak{gl}(\mathfrak{g})$ by
\begin{align}
\rho_R(x, y)z=[Rx,Ry, z]-R([Rx, y, z]+[x, Ry, z])+\lambda[x, y, z], \ \ \ \forall x, y, z \in \mathfrak{g}. \label{5.1}
\end{align}
 Then $(\mathfrak{g},  \rho_R)$ is a representation of the $3$-Lie algebra $(\mathfrak{g}, [\cdot, \cdot, \cdot]_R)$.
\end{lem}
\begin{proof}
For any $x_1, x_2, x_3, x_4, x\in \mathfrak{g}$, by (\ref{3.2}) and (\ref{5.1}),   we have
\begin{small}
\begin{align*}
&\big(\rho_R(x_1, x_2)\circ\rho_R(x_3, x_4)-\rho_R(x_3, x_4)\circ \rho_R(x_1, x_2)-\rho_R([x_1, x_2, x_3]_R, x_4)\\
\quad\quad &-\rho_R(x_3, [x_1, x_2, x_4]_R)\big)(x)\\
&=\rho_R(x_1, x_2)\big([Rx_3,Rx_4, x]-R[Rx_3, x_4, x]-R[x_3, Rx_4, x]+\lambda[x_3, x_4, x]\big)\\
\quad\quad &-\rho_R(x_3, x_4)\big([Rx_1,Rx_2, x]-R[Rx_1, x_2, x]-R[x_1, Rx_2, x]+\lambda[x_1, x_2, x]\big)\\
\quad\quad & -[R[x_1,x_2,x_3]_R, Rx_4, x]+R( [R[x_1,x_2,x_3]_R, x_4, x]+[[x_1, x_2, x_3]_R, Rx_4, x]) \\
\quad\quad &-\lambda [[x_1,x_2,x_3]_R, x_4, x] +[R[x_1,x_2,x_4]_R, Rx_3, x]-R( [R[x_1,x_2,x_4]_R, x_3, x]\\
\quad\quad &+[[x_1, x_2, x_4]_R, Rx_3, x])+\lambda [[x_1,x_2,x_4]_R, x_3, x]\\
&=[Rx_1,Rx_2, [Rx_3,Rx_4, x]]-R([Rx_1, x_2, [Rx_3,Rx_4, x]]+[x_1, Rx_2, [Rx_3,Rx_4, x]])\\
\quad\quad&+[x_1, x_2, [Rx_3,Rx_4, x]]-[Rx_1,Rx_2, R[Rx_3, x_4, x]]+R([Rx_1, x_2, R[Rx_3, x_4, x]]\\
\quad\quad&+[x_1, Rx_2, R[Rx_3, x_4, x]])-[x_1, x_2, R[Rx_3, x_4, x]]-[Rx_1,Rx_2, R[x_3, Rx_4, x]]\\
\quad\quad&+R([Rx_1, x_2, R[x_3, Rx_4, x]]+[x_1, Rx_2, R[x_3, Rx_4, x]])-[x_1, x_2, R[x_3, Rx_4, x]]\\
\quad\quad&+\lambda[Rx_1,Rx_2, [x_3, x_4, x]]-\lambda R[Rx_1, x_2, [x_3, x_4, x]]-\lambda R[x_1, Rx_2, [x_3, x_4, x]]\\
\quad\quad&+\lambda^2[x_1, x_2, [x_3, x_4, x]]_\mathfrak{g}-[Rx_3,Rx_4, [Rx_1,Rx_2, x]]+R([Rx_3, x_4, [Rx_1,Rx_2, x]]\\
\quad\quad&-[x_3, Rx_4, [Rx_1,Rx_2, x]])-[x_3, x_4, [Rx_1,Rx_2, x]]+[Rx_3,Rx_4, R[Rx_1, x_2, x]]\\
\quad\quad&-R([Rx_3, x_4, R[Rx_1, x_2, x]]-[x_3, Rx_4, R[Rx_1, x_2, x]])+[x_3, x_4, R[Rx_1, x_2, x]]\\
\quad\quad&+[Rx_3,Rx_4, R[x_1, Rx_2, x]]-R([Rx_3, x_4, R[x_1, Rx_2, x]]-[x_3, Rx_4, R[x_1, Rx_2, x]])\\
\quad\quad&+[x_3, x_4, R[x_1, Rx_2, x]]-\lambda[Rx_3,Rx_4, [x_1, x_2, x]]+\lambda R[Rx_3, x_4, [x_1, x_2, x]]\\
\quad\quad&+\lambda R[x_3, Rx_4, [x_1, x_2, x]]-\lambda^2[x_3, x_4, [x_1, x_2, x]] -[[Rx_1,Rx_2,Rx_3], Rx_4, x]\\
\quad\quad &-[[Rx_1,x_2,x_3], Rx_4, x]-[[x_1,Rx_2,x_3], Rx_4, x]-[[x_1,x_2,Rx_3], Rx_4, x] \\
\quad\quad &+R[[Rx_1,x_2,x_3], x_4, x]+R[[Rx_1,Rx_2,Rx_3], x_4, x]+R[[x_1,Rx_2,x_3], x_4, x]\\
\quad\quad &+R[[x_1,x_2,Rx_3], x_4, x]+R[[Rx_1, Rx_2, x_3], Rx_4, x] +R[[x_1, Rx_2, Rx_3], Rx_4, x]\\
\quad\quad &+R[[Rx_1, x_2, Rx_3], Rx_4, x]+R[[x_1, x_2, x_3], Rx_4, x]-\lambda [[Rx_1, Rx_2, x_3], x_4, x]\\
\quad\quad &- \lambda[[x_1, Rx_2, Rx_3], x_4, x] - \lambda[[Rx_1, x_2, Rx_3], x_4, x]-\lambda^2 [[x_1, x_2, x_3], x_4, x]\\
\quad\quad &+[[Rx_1,Rx_2,Rx_4], Rx_3, x]+[[Rx_1,x_2,x_4], Rx_3, x]+[[x_1,Rx_2,x_4], Rx_3, x]\\
\quad\quad &+[[x_1,x_2,Rx_4], Rx_3, x] -R[[Rx_1,Rx_2,Rx_4], x_3, x]-R[[Rx_1,x_2,x_4], x_3, x]\\
\quad\quad &-R[[x_1,Rx_2,x_4], x_3, x]-R[[x_1,x_2,Rx_4], x_3, x]-R[[Rx_1, Rx_2, x_4], Rx_3, x] \\
\quad\quad &-R[[x_1, Rx_2, Rx_4], Rx_3, x]-R[[Rx_1, x_2, Rx_4], Rx_3, x]-R[[x_1, x_2, x_4], Rx_3, x]\\
\quad\quad &+\lambda[[Rx_1, Rx_2, x_4], x_3, x]+ \lambda[[x_1, Rx_2, Rx_4], x_3, x] + \lambda[[Rx_1, x_2, Rx_4], x_3, x]\\
\quad\quad &+\lambda^2 [[x_1, x_2, x_4], x_3, x]\\
&=0.
\end{align*}
\end{small}
Similarly, we have
\begin{align*}
\rho_R(x_1, [x_2, x_3, x_4]_R)&=\rho_R(x_3, x_4)\circ \rho_R(x_1, x_2)-\rho_R(x_2, x_4)\circ \rho_R(x_1, x_3)+\rho_R(x_2, x_3)\circ \rho_R(x_1, x_4).
\end{align*}
And the proof is finished.
\end{proof}

Let $\partial_R: C^n_{\mathrm{3Lie}}(\mathfrak{g}) \rightarrow C^{n+1}_{\mathrm{3Lie}}(\mathfrak{g})$ be the corresponding coboundary operator of the $3$-Lie algebra $(\mathfrak{g}, [\cdot, \cdot, \cdot]_R)$ with coefficients in the representation $(\mathfrak{g}, [\cdot, \cdot, \cdot])$. More precisely, $\partial_R: C^n_{\mathrm{3Lie}}(\mathfrak{g}) \rightarrow C^{n+1}_{\mathrm{3Lie}}(\mathfrak{g})$ is given by
\begin{small}
\begin{align*}
&(\partial_R f)(\mathfrak{X}_1, \dots,  \mathfrak{X}_n, x_{n+1})\\
&=\sum_{1\leq j< k\leq n}(-1)^{j}f(\mathfrak{X}_1, \dots, \mathfrak{\hat{X}}_j, \dots, \mathfrak{X}_{k-1}, [x_j, y_j, x_k]_R\wedge y_k
 + x_k\wedge [x_j, y_j, y_k]_R, \mathfrak{X}_{k+1}, \dots, \mathfrak{X}_{n}, x_{n+1})\\
&\quad+ \sum^{n}_{j=1}(-1)^{j-1}\rho_R(x_j, y_j)f(\mathfrak{X}_1, \dots, \mathfrak{\hat{X}}_j, \dots, \mathfrak{X}_{n}, x_{n+1})\\
&\quad+ (-1)^{n+1}(\rho_R(y_n, x_{n+1})f(\mathfrak{X}_1,  \dots, \mathfrak{X}_{n-1}, x_{n})+\rho_R(x_{n+1}, x_n)f(\mathfrak{X}_1,  \dots, \mathfrak{X}_{n-1}, y_{n})),
\end{align*}
\end{small}
for $\mathfrak{X}_i=x_i\wedge y_i\in \wedge^2 \mathfrak{g}, i=1,\cdots, n$ and $x_{n+1}\in \mathfrak{g}$.

Obviously $f\in C^1_{\mathrm{3Lie}}(\mathfrak{g})$ is closed if and only if
\begin{align*}
& [ f(x), Ry, Rz]+[Rx,  f(y), Rz]+[Rx, Ry,  f(z)]\\
&=R( [f(x), Ry, z]+[Rx, f(y), z])+[x, f(y), Rz]+[x, Ry, f(z)]+[Rx, y, f(z)]+[f(x), y, Rz])\\
\quad&\quad+f( [Rx, Ry, z]+[x, Ry, Rz]+[Rx, y, Rz]+\lambda [x, y, z])- \lambda[f(x), y, z]-\lambda[x, f(y), z]-\lambda[x, y, f(z)].
\end{align*}

For any $\mathfrak{X}=x\wedge y\in \wedge^2 \mathfrak{g}$, we define $d_R(\mathfrak{X}) : \mathfrak{g}\rightarrow \mathfrak{g}$ by
\begin{eqnarray*}
d_R(\mathfrak{X})x=R[\mathfrak{X}, x]-[\mathfrak{X}, Rx], \quad \forall \mathfrak{X}\in \wedge^2\mathfrak{g}, x\in \mathfrak{g},
\end{eqnarray*}
\begin{prop}
Let $R$ be a modified Rota-Baxter operator of weight $\lambda$ on a $3$-Lie algebra $(\mathfrak{g}, [\cdot, \cdot, \cdot])$. Then $d_R(\mathfrak{X})$ is a 1-cocycle on the $3$-Lie algebra $(\mathfrak{g}, [\cdot, \cdot, \cdot]_R)$ with coefficients in $(\mathfrak{g}, \rho_R)$.
\end{prop}
\begin{proof}
By direct calculation we can get the conclusion.
\end{proof}
\begin{theorem}\label{Thm: cohomology of modified RB 3 Lie}
Let $R$ be a modified Rota-Baxter operator of weight $\lambda$ on a $3$-Lie algebra $(\mathfrak{g}, [\cdot, \cdot, \cdot]_R)$ with respect to a representation $(\mathfrak{g},  \rho_R)$. Define the set of $n$-cochains by
\begin{align*}
C^n_{\mathrm{R}}(\mathfrak{g} )=
\begin{cases}
C^n_{\mathrm{3Lie}}(\mathfrak{g}), ~~n\geq 2,\\
\mathfrak{g}\wedge \mathfrak{g}, ~~~~~~~~~~ n=1.
\end{cases}
\end{align*}
Define $D_R: C^n_{\mathrm{R}}(\mathfrak{g})\rightarrow C^{n+1}_{\mathrm{R}}(\mathfrak{g} )$ by
\begin{align*}
D_R=
\begin{cases}
\partial_R, ~~n\geq 2,\\
d_R, ~~~~ n=1.
\end{cases}
\end{align*}
\end{theorem}
We denote by
\begin{align*}
    Z^n_R (\mathfrak{g}) = \{ f \in C^n_R (\mathfrak{g}) ~|~ D_R f = 0 \} ~~~ \text{ and } ~~~
    B^n_R (\mathfrak{g}) = \{ D_R g ~|~ g \in C^{n-1}_R(\mathfrak{g}) \}
\end{align*}
the space of $n$-cocycles and $n$-coboundaries, respectively.
The corresponding quotients
\begin{align*}
    H^n_R (\mathfrak{g}) :=   Z^n_R (\mathfrak{g})  /B^n_R (\mathfrak{g}), \text{ for } n \geq 1
\end{align*}
are called the $n$-th cohomology group,  and the cohomology of the cochain complex
$(\oplus_{n=1}^{\infty}C^n_R(\mathfrak{g}), D_R)$ is taken to be the cohomology for the modified Rota-Baxter operator $R$.

\section{Formal deformations  of modified Rota-Baxter operators of weight $\lambda$}\label{Formal deformations  of modified Rota-Baxter operators of weight }
\def\theequation{\arabic{section}.\arabic{equation}}
\setcounter{equation} {0}
In this section, we study  formal  deformations of modified Rota-Baxter operators of weight $\lambda$ on $3$-Lie
algebras.

\begin{defn}
 Let $R$ be a modified Rota-Baxter operator on the $3$-Lie algebra $(\mathfrak{g}, [\cdot, \cdot, \cdot])$ and  $\hat{R}: \mathfrak{g}\rightarrow \mathfrak{g}$ be a
 linear map. If there exists a positive number $\epsilon\in \mathbf{k}$ such that $R_t = R+t \hat{R}$ is still a  modified Rota-Baxter operator
 on the $3$-Lie algebra $(\mathfrak{g}, [\cdot, \cdot, \cdot])$ for all $t\in (-\epsilon, \epsilon)$, we say that $\hat{R}$ generates a linear deformation of
 the  modified Rota-Baxter operator $R$.
\end{defn}

\begin{defn}
Let $R$ be a modified Rota-Baxter operator on the $3$-Lie algebra $(\mathfrak{g}, [\cdot, \cdot, \cdot])$. Two linear deformations $R^1_t = R+t\hat{R}_1$ and $R^2_t = R+t\hat{R}_2$ are said to be equivalent if there exists an $\mathfrak{X}\in \mathfrak{g}\wedge \mathfrak{g}$ such that
 \begin{align*}
   \varphi_t=\mathrm{id} + t\mathrm{ad}_{\mathfrak{X}},
\end{align*}
 satisfies the following conditions:
 \begin{align*}
&(a) \quad \varphi_t([x, y, z])=[ \varphi_t(x), \varphi_t(y),  \varphi_t(z)], \quad \forall x, y, z\in \mathfrak{g},\\
&(b) \quad R_t^2\circ  \varphi_t = \varphi_t \circ R^1_t,
\end{align*}
\end{defn}
\begin{theorem}
 Let $\hat{R}:\mathfrak{g}\rightarrow \mathfrak{g}$ generate a linear deformation of the modified Rota-Baxter operator $R$. Then $\hat{R}$ is a 1-cocycle.\\
 Let $R^1_t$ and $R^2_t$ be equivalent linear deformations of $R$ generated by $\hat{R}_1$ and $\hat{R}_2$ respectively. Then $\hat{R}_1$ and $\hat{R}_2$ belong to the same cohomology class $H^2_R (\mathfrak{g})$.
\end{theorem}
\begin{proof}
Since $R_t = R+t\hat{R}$ are modified Rota-Baxter operators,  for any  $x, y, z \in \mathfrak{g}$, we have
\begin{align*}
& [R_tx, R_ty, R_tz]\\
 &=R_t( [R_tx, R_ty, z]+[x, R_ty, R_tz]_\mathfrak{g}+[R_tx, y, R_tz]+\lambda [x, y, z])\\
\quad &\quad-\lambda [R_tx, y, z]-\lambda[x, R_ty, z]-\lambda[x, y, R_tz],
\end{align*}
which implies
\begin{align}
& [ \hat{R}x, Ry, Rz]+[Rx,  \hat{R}y, Rz]+[Rx, Ry,  \hat{R}z]\label{6.1}\\
&=R( [\hat{R}x, Ry, z]+[Rx, \hat{R}y, z])+[x, \hat{R}y, Rz]\nonumber\\
\quad&\quad+[x, Ry, \hat{R}z]+[Rx, y, \hat{R}z]+[\hat{R}x, y, Rz])\nonumber\\
\quad&\quad+\hat{R}( [Rx, Ry, z]+[x, Ry, Rz]+[Rx, y, Rz]_\mathfrak{g}+\lambda[x, y, z])\nonumber\\
\quad &\quad- \lambda[\hat{R}x, y, z]-\lambda[x, \hat{R}y, z]-\lambda[x, y, \hat{R}z],\nonumber
\end{align}
\begin{align}
& [ Rx, \hat{R}y, \hat{R}z]+[\hat{R}x,  \hat{R}y, Rz]+[ \hat{R}x, Ry,  \hat{R}z]\\
\quad&=\hat{R}([\hat{R}x, Ry, z]+[Rx, \hat{R}y, z]+[x, \hat{R}y, Rz]\nonumber\\
\quad&\quad +[x, Ry, \hat{R}z]+[\hat{R}x, y, Rz]+[Rx, y, \hat{R}z]\nonumber\\
\quad&\quad+R( [\hat{R}x, \hat{R}y, z]+[x, \hat{R}y, \hat{R}z]+[\hat{R}x, y, \hat{R}z]),\nonumber
\end{align}
and
\begin{align}
& [\hat{R}x, \hat{R}y, \hat{R}z]=\hat{R}( [\hat{R}x, \hat{R}y, z]+[x, \hat{R}y, \hat{R}z]+[\hat{R}x, y, \hat{R}z]).\label{6.3}
\end{align}
By (\ref{6.1}),  $\hat{R}$ is a 1-cocycle.

 If $R^1_t$ and $R^2_t$ are equivalent linear deformations of $R$, then there exists $\mathfrak{X}\in \mathfrak{g}\wedge \mathfrak{g}$ such that
 \begin{align*}
(\mathrm{id}+t\mathrm{ad}_{\mathfrak{X}})(R+t\hat{R}_1)(u)=(R+t\hat{R}_2)(\mathrm{id}+t\mathrm{ad}_{\mathfrak{X}}) (u), \ \ \ \forall u\in \mathfrak{g},
\end{align*}
which implies
\begin{align}
\hat{R}_1(u)-\hat{R}_2(u)= R[\mathfrak{X}, u]-[\mathfrak{X}, Ru], \ \ \ \ \ \forall u\in \mathfrak{g},\label{6.4}
\end{align}
by(\ref{6.4}), we obtain
\begin{align*}
\hat{R}_1-\hat{R}_2=\partial_R(x).
\end{align*}
Thus, $\hat{R}_1$ and $\hat{R}_2$ belong to the same cohomology class $H^2_R (\mathfrak{g})$.
\end{proof}
\begin{defn}
A linear deformation of a modified Rota-Baxter operator $R$ generated by $\hat{R}$  is trivial if there
 exists an $\mathfrak{X}\in \mathfrak{g}\wedge \mathfrak{g}$ such that $\mathrm{id}_\mathfrak{g}+t\mathrm{ad}_{\mathfrak{X}}$ is an isomorphism from  $R_t = R + t\hat{R}$ to $R$.
\end{defn}

\begin{defn}
Let $R$ be a modified Rota-Baxter operator on the $3$-Lie algebra $(\mathfrak{g}, [\cdot, \cdot, \cdot])$. An element $\mathfrak{X}\in \mathfrak{g}\wedge \mathfrak{g}$  is called
 a Nijenhuis element associated to $R$ if $\mathfrak{X}$ satisfies
 \begin{align}
& [[\mathfrak{X}, x],[\mathfrak{X}, y],[\mathfrak{X}, z]]=0,\label{6.5}\\
& [x, [\mathfrak{X}, y],[\mathfrak{X}, z]]+[[\mathfrak{X}, x], [\mathfrak{X}, y], z]+[[\mathfrak{X}, x],y,[\mathfrak{X}, z]]=0,\label{6.6}\\
& [\mathfrak{X}, R[\mathfrak{X}, x]] =[\mathfrak{X}, [\mathfrak{X}, Rx]].\label{6.7}
\end{align}
 \end{defn}
 \begin{theorem}
 Let $R$ be a modified Rota-Baxter operator on the $3$-Lie algebra $(\mathfrak{g}, [\cdot, \cdot, \cdot])$. Then for any Nijenhuis element
$\mathfrak{X}\in \mathfrak{g}\wedge \mathfrak{g}$, $R_t =R+t\partial_R(\mathfrak{X})$ is a trivial linear deformation of the modified Rota-Baxter operator $R$.
\end{theorem}
\begin{proof}
Taking $\hat{R}=\partial_R(\mathfrak{X})$, for any  $x, y,z \in \mathfrak{g}$, we have  $\hat{R}x=R[\mathfrak{X},x]-[\mathfrak{X},Rx].$ To show that $ R_t$ is a linear deformation of $R$, we need to
 verify that (\ref{6.3}) holds, in fact, we have
 \begin{align*}
&\hat{R}( [\hat{R}x, \hat{R}y, z]+[x, \hat{R}y, \hat{R}z]+[\hat{R}x, y, \hat{R}z])-[\hat{R}x, \hat{R}y, \hat{R}z]\\
&=R[\mathfrak{X},[R[\mathfrak{X},x]-[\mathfrak{X},Rx], R[\mathfrak{X},y]-[\mathfrak{X},Ry], z]]\\
&\quad-[\mathfrak{X},R[R[\mathfrak{X},x]-[\mathfrak{X},Rx], R[\mathfrak{X},y]-[\mathfrak{X},Ry], z]]\\
&\quad+R[\mathfrak{X},[x, R[\mathfrak{X},y]-[\mathfrak{X},Ry], R[\mathfrak{X},z]-[\mathfrak{X},Rz]]]\\
&\quad-[\mathfrak{X},R[x, R[\mathfrak{X},y]-[\mathfrak{X},Ry], R[\mathfrak{X},z]-[\mathfrak{X},Rz]]]\\
&\quad+R[\mathfrak{X},[R[\mathfrak{X},x]-[\mathfrak{X},Rx], y, R[\mathfrak{X},z]-[\mathfrak{X},Rz]])]\\
&\quad-[\mathfrak{X},R[R[\mathfrak{X},x]-[\mathfrak{X},Rx], y, R[\mathfrak{X},z]-[\mathfrak{X},Rz]])]\\
&\quad-[R[\mathfrak{X},x]-[\mathfrak{X},Rx], R[\mathfrak{X},y]-[\mathfrak{X},Ry], R[\mathfrak{X},z]-[\mathfrak{X},Rz]]\\
&=0.
\end{align*}
 Thus, $ R_t$ is a linear deformation of the modified Rota-Baxter operator $R$. Since $\mathfrak{X}$ is a
 Nijenhuis element, we have $(\mathrm{id}+t\mathrm{ad}_{\mathfrak{X}})[x, y, z]=[x+t[\mathfrak{X}, x], y+t[\mathfrak{X}, y], z+t[\mathfrak{X}, z]]$ and $
R(x + t[\mathfrak{X}, x]) =(\mathrm{id}+t\mathrm{ad}_{\mathfrak{X}})(Rx+t\hat{R}x)$. Thus,   $R_t =R+t\partial_R(\mathfrak{X})$ is a trivial linear
 deformation of the modified Rota-Baxter operator $R$.
\end{proof}

\begin{prop}
 Let $\hat{R}$ generate a linear deformation of a modified Rota-Baxter operator $R$ on a $3$-Lie algebra
$(\mathfrak{g}, [\cdot, \cdot, \cdot])$. Then $\omega$ defined by
\begin{align*}
\omega(x, y, z)&=[Rx, \hat{R}y, z]+[\hat{R}x, Ry, z]+[Rx, y, \hat{R}z]+[\hat{R}x, y, Rz]\\
&\quad+[x, \hat{R}y, Rz]+[x, Ry, \hat{R}z],
\end{align*}
generates a linear deformation of the $3$-Lie algebra $(\mathfrak{g}, [\cdot, \cdot, \cdot]_R)$ given by the modified Rota-Baxter operator $R$.
\end{prop}
\begin{proof}
For any  $x, y,z \in \mathfrak{g}$,  we have
\begin{align*}
[x, y, z]_{R_t}&=[Rx, Ry, z]+[x, Ry, Rz]+[Rx, y, Rz]+\lambda[x, y, z]+t([Rx, \hat{R}y, z]\\
\quad&\quad+[\hat{R}x, Ry, z]+[Rx, y, \hat{R}z]+[\hat{R}x, y, Rz]\\
\quad&\quad+[x, \hat{R}y, Rz]+[x, Ry, \hat{R}z])\\
&=[x, y, z]_{\mathrm{R}}+t\omega(x, y, z).
\end{align*}
 Since $(\mathfrak{g}, [\cdot, \cdot, \cdot]_{R_t})$ are $3$-Lie algebras, we have that $\omega$ generates a linear deformation of the $3$-Lie
 algebra $(\mathfrak{g}, [\cdot, \cdot, \cdot]_R)$ given by the modified Rota-Baxter operator $R$.
\end{proof}
 \begin{theorem}
Let $\mathfrak{X}\in \mathfrak{g}\wedge \mathfrak{g}$  be a Nijenhuis element associated to a modified Rota-Baxter operator $R$. Then $\mathrm{ad}_\mathfrak{X}$ is a
 Nijenhuis operator on the $3$-Lie algebra $(\mathfrak{g}, [\cdot, \cdot, \cdot]_R)$.
\end{theorem}
\begin{proof}
For any $x, y,z \in \mathfrak{g}$, by (\ref{6.5})-(\ref{6.7}), we have
\begin{align*}
&[\mathrm{ad}_\mathfrak{X}x, \mathrm{ad}_\mathfrak{X}y, \mathrm{ad}_\mathfrak{X}z]_R\\
&=[R[\mathfrak{X}, x], R[\mathfrak{X}, y], z]+[x, R[\mathfrak{X}, y], R[\mathfrak{X}, z]]+[R[\mathfrak{X}, x], y, R[\mathfrak{X}, z]]\\
\quad&\quad+\lambda[[\mathfrak{X}, x], [\mathfrak{X}, y], [\mathfrak{X}, z]]\\
&=[R[\mathfrak{X}, x], R[\mathfrak{X}, y], z]+[x, R[\mathfrak{X}, y], R[\mathfrak{X}, z]]+[R[\mathfrak{X}, x], y, R[\mathfrak{X}, z]].
\end{align*}
On the other hand,
\begin{align*}
&\mathrm{ad}_\mathfrak{X}\big( [\mathrm{ad}_\mathfrak{X}x, \mathrm{ad}_\mathfrak{X}y, z]_R+[x, \mathrm{ad}_\mathfrak{X}y, \mathrm{ad}_\mathfrak{X}z]_R+[\mathrm{ad}_\mathfrak{X}x, y, \mathrm{ad}_\mathfrak{X}z]_R\\
\quad & \quad -\mathrm{ad}_\mathfrak{X}([\mathrm{ad}_\mathfrak{X}x, y, z]_R+[x, \mathrm{ad}_\mathfrak{X}y, z]_R+[x, y, \mathrm{ad}_\mathfrak{X}z]_R-\mathrm{ad}_\mathfrak{X}[x, y, z]_R)\big)\\
\quad&=\mathrm{ad}_\mathfrak{X}\Big( [[\mathfrak{X}, x], [\mathfrak{X}, y], z]_R+[x, [\mathfrak{X}, y], [\mathfrak{X}, z]]_R+[[\mathfrak{X}, x], y, [\mathfrak{X}, z]]_R\\
\quad&\quad -\mathrm{ad}_\mathfrak{X}([[\mathfrak{X}, x], y, z]_R+[x, [\mathfrak{X}, y], z]_R+[x, y, [\mathfrak{X}, z]]_R-\mathrm{ad}_\mathfrak{X}[x, y, z]_R)\Big)\\
&=\mathrm{ad}_\mathfrak{X}\Big( [R[\mathfrak{X}, x], R[\mathfrak{X}, y], z]+[[\mathfrak{X}, x], R[\mathfrak{X}, y], Rz]+[R[\mathfrak{X}, x], [\mathfrak{X}, y], Rz]\\
\quad&\quad+[Rx, R[\mathfrak{X}, y], [\mathfrak{X}, z]]+[x, R[\mathfrak{X}, y], R[\mathfrak{X}, z]]+[Rx, [\mathfrak{X}, y], R[\mathfrak{X}, z]]\\
\quad&\quad+[R[\mathfrak{X}, x], Ry, [\mathfrak{X}, z]]+[[\mathfrak{X}, x], Ry, R[\mathfrak{X}, z]]+[R[\mathfrak{X}, x], y, R[\mathfrak{X}, z]]\\
\quad&\quad-\mathrm{ad}_\mathfrak{X}([R[\mathfrak{X}, x], Ry, z]+[R[\mathfrak{X}, x], y, Rz]+[[\mathfrak{X}, x], Ry, Rz]+[Rx, R[\mathfrak{X}, y], z]\\
\quad&\quad+[x, R[\mathfrak{X}, y], Rz]+[Rx, [\mathfrak{X}, y], Rz]+[Rx, Ry, [\mathfrak{X}, z]]+[x, Ry, R[\mathfrak{X}, z]]\\
\quad&\quad +[Rx, y, R[\mathfrak{X}, z]]-[[\mathfrak{X},Rx], Ry, z] -[Rx, [\mathfrak{X},Ry], z]-[Rx, Ry, [\mathfrak{X},z]]\\
\quad&\quad-[[\mathfrak{X}, x], Ry, Rz]-[x, [\mathfrak{X}, Ry], Rz]-[x, Ry, [\mathfrak{X}, Rz]]\\
\quad&\quad-[[\mathfrak{X},Rx], y, Rz]-[Rx, [\mathfrak{X},y], Rz]-[Rx, y, [\mathfrak{X},Rz]])\Big)\\
&=\mathrm{ad}_\mathfrak{X}\Big( [R[\mathfrak{X}, x], R[\mathfrak{X}, y], z]+[Rx, R[\mathfrak{X}, y], [\mathfrak{X}, z]]\\
\quad&\quad+[x, R[\mathfrak{X}, y], R[\mathfrak{X}, z]]+[R[\mathfrak{X}, x], y, R[\mathfrak{X}, z]]-[ R[\mathfrak{X}, x]],[\mathfrak{X}, Ry], z]\\
\quad&\quad-[R[\mathfrak{X}, x], y, [\mathfrak{X}, Rz]]-[Rx, R[\mathfrak{X}, y]], [\mathfrak{X},z]]\\
\quad&\quad-[ x, R[\mathfrak{X}, y], [\mathfrak{X},Rz]]-[ x, [\mathfrak{X},Ry], R[\mathfrak{X}, z]]-[ [\mathfrak{X},Rx], y, R[\mathfrak{X}, z]]\\
\quad&\quad+[ [\mathfrak{X},Rx], [\mathfrak{X}, Ry], z]+[x,[\mathfrak{X}, Ry], [\mathfrak{X}, Rz]]+[Rx, [\mathfrak{X},y], [\mathfrak{X},Rz]]\Big)\\
&=[[\mathfrak{X},R[\mathfrak{X}, x]], R[\mathfrak{X}, y], z]+[[\mathfrak{X},Rx], R[\mathfrak{X}, y], [\mathfrak{X}, z]]+[[\mathfrak{X}, x], R[\mathfrak{X}, y], R[\mathfrak{X}, z]]\\
&\quad+[[\mathfrak{X}, R[\mathfrak{X}, x]], Ry, [\mathfrak{X}, z]]-[ [\mathfrak{X}, R[\mathfrak{X}, x]],[\mathfrak{X}, Ry], z]-[[\mathfrak{X}, R[\mathfrak{X}, x]],[\mathfrak{X}, y], Rz]\\
&\quad-[[\mathfrak{X}, R[\mathfrak{X}, x]], y, [\mathfrak{X}, Rz]]-[[\mathfrak{X}, [\mathfrak{X}, Rx]], R[\mathfrak{X}, y], z]-[[\mathfrak{X}, Rx], R[\mathfrak{X}, y]], [\mathfrak{X},z]]\\
&\quad-[[\mathfrak{X}, x], R[\mathfrak{X}, y], [\mathfrak{X},Rz]]-[[\mathfrak{X}, x], [\mathfrak{X},Ry], R[\mathfrak{X}, z]]-[[\mathfrak{X}, [\mathfrak{X},Rx]], y, R[\mathfrak{X}, z]]\\
&\quad+[[\mathfrak{X}, [\mathfrak{X},Rx]], [\mathfrak{X}, Ry], z]\\
&=[[\mathfrak{X},Rx], R[\mathfrak{X}, y], [\mathfrak{X}, z]]+[[\mathfrak{X}, x], R[\mathfrak{X}, y], R[\mathfrak{X}, z]]+[[\mathfrak{X}, R[\mathfrak{X}, x]], Ry, [\mathfrak{X}, z]]\\
&\quad-[[\mathfrak{X}, R[\mathfrak{X}, x]], y, [\mathfrak{X}, Rz]]-[[\mathfrak{X}, Rx], R[\mathfrak{X}, y]], [\mathfrak{X},z]]-[[\mathfrak{X}, x], R[\mathfrak{X}, y], [\mathfrak{X},Rz]]\\
&\quad-[[\mathfrak{X}, x], [\mathfrak{X},Ry], R[\mathfrak{X}, z]]-[[\mathfrak{X}, R[\mathfrak{X}, x]], y, R[\mathfrak{X}, z]]\\
&=[R[\mathfrak{X}, x], R[\mathfrak{X}, y], z]+[x, R[\mathfrak{X}, y], R[\mathfrak{X}, z]]+[R[\mathfrak{X}, x], y, R[\mathfrak{X}, z]].
\end{align*}
Thus, we have
\begin{align*}
&[\mathrm{ad}_\mathfrak{X}x, \mathrm{ad}_\mathfrak{X}y, \mathrm{ad}_\mathfrak{X}z]_R\\
&=\mathrm{ad}_\mathfrak{X}\big( [\mathrm{ad}_\mathfrak{X}x, \mathrm{ad}_\mathfrak{X}y, z]_R+[x, \mathrm{ad}_\mathfrak{X}y, \mathrm{ad}_\mathfrak{X}z]_R+[\mathrm{ad}_\mathfrak{X}x, y, \mathrm{ad}_\mathfrak{X}z]_R\\
\quad &  -\mathrm{ad}_\mathfrak{X}([\mathrm{ad}_\mathfrak{X}x, y, z]_R+[x, \mathrm{ad}_\mathfrak{X}y, z]_R+[x, y, \mathrm{ad}_\mathfrak{X}z]_R-\mathrm{ad}_\mathfrak{X}[x, y, z]_R)\big).
\end{align*}
 which implies that $\mathrm{ad}_\mathfrak{X}$ is a
 Nijenhuis operator on the $3$-Lie algebra $(\mathfrak{g}, [\cdot, \cdot, \cdot]_R)$.
\end{proof}

\section{$L_\infty[1]$-structure for (relative and absolute) modified Rota-Baxter 3-Lie algebras}\label{L infinty-structure for (relative and absolute) modified Rota-Baxter 3-Lie algebras}

\subsection{$L_\infty[1]$-algebras and (generalised) derived brackets}
In this subsection, we recall some preliminaries on $L_\infty[1]$-algebras.
\begin{defn}
	An $L_\infty[1]$-algebra is a pair $(L = \oplus_{i \in \mathbb{Z}} L_i, \{ l_k \}_{k=1}^\infty)$ consisting of a graded vector space $L = \oplus_{i \in \mathbb{Z}} L_i$ together with a collection $\{ l_k \}_{k=1}^\infty$ of degree $1$ multilinear maps $l_k : L^{\otimes k} \rightarrow L$ (for $k \geq 1$) that satisfy

	- graded symmetry: ~~$l_k(x_{\sigma(1)}\otimes\cdots\otimes x_{\sigma(k)})=\epsilon(\sigma)l_k(x_1\otimes\cdots\otimes x_k),$ for any $\sigma\in S_k$ and $k \geq 1$,
	
	- higher Jacobi identities:
	\begin{align*}
		\sum_{i+j=n+1}\sum_{\sigma\in S_{(i,n-i)}}\epsilon(\sigma)l_j\left(l_i\left(x_{\sigma(1)}\otimes\cdots\otimes x_{\sigma(i)} \right) \otimes x_{\sigma(i+1)}\otimes \cdots\otimes x_{\sigma(n) }\right)=0,
	\end{align*}
	for $x_1, \ldots, x_n \in L$ and $n \geq 1$. Here $S_{(i,n-i)}$ is the set of all $(i, n-i)$-shuffles i.e., permutations \( \sigma \in S_n \) such that:
	\[
	\sigma(1) <   \cdots <   \sigma(i_1),
	\sigma(i_1+1) <   \cdots <   \sigma(i_1 + i_2),  \dots,
	\sigma(i_1 + \cdots + i_{r-1} + 1) <   \cdots <   \sigma(n).
	\] and $\epsilon (\sigma) = \epsilon (\sigma; x_1, \ldots, x_n)$ is the standard Koszul sign in the graded context.
\end{defn}

Throughout the paper, all $L_\infty[1]$-algebras are assumed to be weakly filtered. In other words, certain infinite summations are always convergent.

\begin{defn}
	Let $(L = \oplus_{i \in \mathbb{Z}} L_i, \{ l_k \}_{k=1}^\infty)$  be an $L_\infty[1]$-algebra. An element $\alpha \in L_0$ is said to be a Maurer-Cartan element if it satisfies
	\begin{align*}
		\sum_{k=1}^\infty \frac{1}{k !} ~\! l_k (\alpha\otimes \cdots\otimes \alpha) = 0.
	\end{align*}
\end{defn}

\begin{prop}\cite[Twisting procedure]{Get09}\label{Twisting procedure}
	Let $\alpha$ be a Maurer-Cartan element of an $L_\infty[1]$-algebra $L$. The twisted $L_\infty[1]$-algebra structure on $L$ is given by a sequence of multilinear maps $l_n^\alpha: L^{\otimes n} \to L$ defined by
	$$
	l_n^\alpha\left(x_1 \otimes \cdots \otimes x_n\right)=\sum_{i=0}^{\infty} \frac{1}{i!} l_{n+i}\left( \underbrace{\alpha\otimes \cdots\otimes \alpha}_{i \text{ times}}\otimes x_1 \otimes \cdots\otimes  x_n \right), \quad \forall x_1, \ldots, x_n \in L,
	$$ whenever the infinite sum converges (or terminates in a finite case)
	
	In particular, a Maurer-Cartan element $\alpha \in L_0$ induces a degree-one differential
	\begin{align*}
		l_1^\alpha (x) = \sum_{i=0}^\infty \frac{1}{i !} ~ \! l_{i+1} (\underbrace{\alpha\otimes \cdots\otimes \alpha}_{i \text{ times}}\otimes x),
	\end{align*}
	which turns $(L, l_1^\alpha)$ into a cochain complex. The corresponding cohomology groups are referred to as the cohomology induced by the Maurer-Cartan element $\alpha$.

\end{prop}

\medskip

An important class of $L_\infty[1]$-algebras arise from $V$-datas \cite{V05}. Recall that a $V$-data is a quadrupe $(V, \mathfrak{a}, \mathcal{P}, \triangle)$ in which
\begin{itemize}
	\item [(a)]$V$ is a graded Lie algebra (with the graded Lie bracket $[~,~]$);
	\item [(b)]$\mathfrak{a} \subset V$ is an abelian graded Lie subalgebra;
	\item [(c)]$\mathcal{P} : V \rightarrow \mathfrak{a}$ is a projection map with the property that $\mathrm{ker}(\mathcal{P}) \subset V$ is a graded Lie subalgebra;
	\item [(d)]and $\triangle \in \mathrm{ker}(\mathcal{P})^1$ that satisfies $[\triangle, \triangle] = 0$.
\end{itemize}


The construction of $L_\infty$-subalgebras using $V$-data has been studied in \cite{DM22, FZ15, LQYZ24}. We summarize this method in the following simplified theorem.

\begin{theorem}\label{v-th}
Let $(V, \mathfrak{a}, \mathcal{P}, \triangle)$ be a $V$-data. Suppose $V' \subset V$ is a graded Lie subalgebra that satisfies $[\triangle, V'] \subset V'$. Then the graded vector space $V'[1] \oplus \mathfrak{a}$ carries an $L_\infty[1]$-algebra structure with the multilinear operations
\begin{align*}
	l_1 (x[1] \otimes a) &= \big( -[\triangle, x][1],~ \mathcal{P}(x + [\triangle, a]) \big),\\
	l_2 (x[1] \otimes y[1]) &= (-1)^{|x|} [x, y][1],\\
	l_k (x[1] \otimes a_1 \otimes \cdots \otimes a_{k-1}) &= \mathcal{P} \left[ \cdots [[x, a_1], a_2], \ldots, a_{k-1} \right], \quad \text{for } k \geq 2,\\
	l_k (a_1 \otimes \cdots \otimes a_k) &= \mathcal{P} \left[ \cdots [[\triangle, a_1], a_2], \ldots, a_k \right], \quad \text{for } k \geq 2.
\end{align*}
Here $x, y$ are homogeneous elements in $V'$ and $a_1, \ldots, a_k$ are homogeneous elements in $\mathfrak{a}$. Up to the permutations of the above entries, all other multilinear operations vanish.

Moreover, let $\iota: V'' \hookrightarrow V'$ be a monomorphism of graded Lie algebras such that $[\triangle, \iota(V'')] \subset \iota(V'')$. Then the graded vector space $V''[1] \oplus \mathfrak{a}$ also carries an $L_\infty[1]$-algebra structure with the multilinear operations:
\begin{align*}
	l_1' (x[1] \otimes a) &= \big( -\iota^{-1}[\triangle, \iota(x)][1],~ \mathcal{P}(\iota(x) + [\triangle, a]) \big),\\
	l_2' (x[1] \otimes y[1]) &= (-1)^{|x|} \iota^{-1}[\iota(x), \iota(y)][1],\\
	l_k' (x[1] \otimes a_1 \otimes \cdots \otimes a_{k-1}) &= \mathcal{P} \left[ \cdots [[\iota(x), a_1], a_2], \ldots, a_{k-1} \right], \quad \text{for } k \geq 2,\\
	l_k' (a_1 \otimes \cdots \otimes a_k) &= \mathcal{P} \left[ \cdots [[\triangle, a_1], a_2], \ldots, a_k \right], \quad \text{for } k \geq 2.
\end{align*}
Here $x, y$ are homogeneous elements in $V''$ and $a_1, \ldots, a_k$ are homogeneous elements in $\mathfrak{a}$.

In addition, the map $\iota$ induces a monomorphism of $L_\infty[1]$-algebras:
\[
\tilde{\iota} : \big(V_{\mathfrak{h}}[1] \oplus \mathfrak{a}, \{l^{\mathrm{RB}}_k\}_{k \geq 1} \big)
\longrightarrow
\big(V_{r\mathrm{Pair}}[1] \oplus \mathfrak{a}, \{l^{m\mathrm{RB}}_k\}_{k \geq 1} \big), \quad
(f[1], \theta) \mapsto (\iota(f)[1], \theta).
\]
\end{theorem}

\subsection{$L_\infty[1]$-structure for relative  modified Rota-Baxter 3-Lie algebras}\label{L infty  structure for relative  modified Rota-Baxter 3-Lie algebras}In this subsection, by using the derived bracket technique, we construct an
$L_\infty[1]$-algebra whose Maurer-Cartan elements are in bijection with the set of structures of relative
modified Rota-Baxter $3$-Lie algebras of weight $\lambda$.

\begin{theorem}\label{Thm: graded Lie structure of 3-Lie complex}\cite{Rotkiewicz}
	The graded vector space $C^*_{\mathrm{3Lie}}(\mathfrak{g}, \mathfrak{g})$ equipped with the graded commutator bracket
	
	$$
	[P, Q]_{\mathrm{R}}=P \circ Q-(-1)^{p q} Q \circ P, \quad \forall P \in C^p_{\mathrm{3Lie}}(\mathfrak{g}, \mathfrak{g}), Q \in C^q_{\mathrm{3Lie}}(\mathfrak{g}, \mathfrak{g}),
	$$
	
	is a graded Lie algebra, where $P \circ Q \in C^{p+q}_{\mathrm{3Lie}}(\mathfrak{g}, \mathfrak{g})$ is defined by
	\begin{small}
	$$
	\begin{aligned}
		& (P \circ Q)\left(\mathfrak{X}_1, \cdots, \mathfrak{X}_{p+q}, x\right) \\
		= & \sum_{k=1}^p(-1)^{(k-1) q} \sum_{\sigma \in \mathbb{S}(k-1, q)}(-1)^\sigma P\left(\mathfrak{X}_{\sigma(1)}, \cdots, \mathfrak{X}_{\sigma(k-1)}, Q\left(\mathfrak{X}_{\sigma(k)}, \cdots, \mathfrak{X}_{\sigma(k+q-1)}, x_{k+q}\right) \wedge y_{k+q}, \mathfrak{X}_{k+q+1}, \cdots, \mathfrak{X}_{p+q}, x\right) \\
		& +\sum_{k=1}^p(-1)^{(k-1) q} \sum_{\sigma \in \mathbb{S}(k-1, q)}(-1)^\sigma P\left(\mathfrak{X}_{\sigma(1)}, \cdots, \mathfrak{X}_{\sigma(k-1)}, x_{k+q} \wedge Q\left(\mathfrak{X}_{\sigma(k)}, \cdots, \mathfrak{X}_{\sigma(k+q-1)}, y_{k+q}\right), \mathfrak{X}_{k+q+1}, \cdots, \mathfrak{X}_{p+q}, x\right) \\
		& +\sum_{\sigma \in \mathbb{S}(p, q)}(-1)^{p q}(-1)^\sigma P\left(\mathfrak{X}_{\sigma(1)}, \cdots, \mathfrak{X}_{\sigma(p)}, Q\left(\mathfrak{X}_{\sigma(p+1)}, \cdots, \mathfrak{X}_{\sigma(p+q-1)}, \mathfrak{X}_{\sigma(p+q)}, x\right)\right),
	\end{aligned}
	$$
	\end{small}
	for all $\mathfrak{X}_i=x_i \wedge y_i \in \wedge^2 \mathfrak{g}, i=1,2, \cdots, p+q$ and $x \in \mathfrak{g}$.
	Moreover, $\mu: \wedge^3 \mathfrak{g} \longrightarrow \mathfrak{g}$ is a 3-Lie bracket if and only if  $\mu\in C^{2}_{\mathrm{3Lie}}(\mathfrak{g}, \mathfrak{g})$ such that $[\mu,\mu]_{\mathrm{R}}=0$.
\end{theorem}
\begin{prop}
	\label{Prop:subalgebra}
	Denote the graded vector space \( C^*_{\mathrm{3Lie\text{-}act}}(\mathfrak{h}, \mathfrak{g}) \) by
\[
\bigoplus_{n \geq 0} \bigoplus_{i_1, \ldots, i_n \in \{0,1\}} \left(
\begin{aligned}
	&\operatorname{Hom}\left(
	V_{i_1} \otimes \cdots \otimes V_{i_n} \otimes \big( \wedge^3\mathfrak{g} \oplus (\wedge^2\mathfrak{h} \otimes \mathfrak{g}) \big),
	\, \mathfrak{g}
	\right) \\
	&\quad\oplus\
	\operatorname{Hom}\left(
	V_{i_1} \otimes \cdots \otimes V_{i_n} \otimes \big( \wedge^3\mathfrak{h} \oplus (\wedge^2\mathfrak{g} \otimes \mathfrak{h}) \big),
	\, \mathfrak{h}
	\right)
\end{aligned}
\right)
\]

	where \( V_0 = \wedge^2 \mathfrak{h} \) and \( V_1 = \wedge^2 \mathfrak{g} \). Then \( C^*_{\mathrm{3Lie\text{-}act}}(\mathfrak{h}, \mathfrak{g}) \) is a subalgebra of the graded Lie algebra
	\[
	\left( C^*_{\mathrm{3Lie}}(\mathfrak{g} \oplus \mathfrak{h}, \mathfrak{g} \oplus \mathfrak{h}), [\cdot, \cdot]_{\mathrm{R}} \right).
	\]
	
	Moreover, a quadruple
	\[
	\left( \left(\mathfrak{h}, \mu\right), \left(\mathfrak{g}, \pi\right), \rho, \zeta \right)
	\]
	forms a  relative \(3\)-Lie algebra pair  if and only if \( \delta := \pi + \rho + \mu + \zeta \in (V_{r\mathrm{Pair}})^1 \) satisfies
	\[
	[\delta, \delta]_{\mathrm{R}} = 0,
	\]
	where
	\[
	\pi \in \operatorname{Hom}(\wedge^3 \mathfrak{g}, \mathfrak{g}), \quad
	\mu \in \operatorname{Hom}(\wedge^3 \mathfrak{h}, \mathfrak{h}), \quad
	\rho \in \operatorname{Hom}(\wedge^2 \mathfrak{g} \otimes \mathfrak{h}, \mathfrak{h}), \quad
	\zeta \in \operatorname{Hom}(\wedge^2 \mathfrak{h} \otimes \mathfrak{g}, \mathfrak{g}).
	\]
\end{prop}
\begin{proof}
	Denote $$\bigoplus_{n \geq 0} \bigoplus_{i_1, \ldots, i_n \in \{0,1\}} \left(
	\operatorname{Hom}(V_{i_1} \otimes \cdots \otimes V_{i_n} \otimes \left( (\wedge^3\mathfrak{g})\oplus (\wedge^2\mathfrak{h}\otimes\mathfrak{g})\right) , \mathfrak{g})
	\right)$$ by $V_{\mathfrak{g}}$ and $$ \bigoplus_{n \geq 0} \bigoplus_{i_1, \ldots, i_n \in \{0,1\}} \left(
	\operatorname{Hom}(V_{i_1} \otimes \cdots \otimes V_{i_n} \otimes \left( (\wedge^3\mathfrak{h})\oplus (\wedge^2\mathfrak{g}\otimes\mathfrak{h})\right) , \mathfrak{h})
	\right)$$ by $V_{\mathfrak{h}}$, respectively.  For each $P\in V_{\mathfrak{g}}$, $Q\in V_{\mathfrak{h}} $ and $R\in V_{r\mathrm{Pair}}$, we can show that $P \circ R\in V_{\mathfrak{g}}$ and $Q\circ R\in V_{\mathfrak{h}}$. Thus,  $V_{r\mathrm{Pair}}$ is a subalgebra of $ C^*_{\mathrm{3Lie}}(\mathfrak{g}\oplus\mathfrak{h}, \mathfrak{g}\oplus\mathfrak{h})$.

	For all $\delta\in (V_{r\mathrm{Pair}})^1$,  can decompose
	$\delta$ as $\delta=\pi+\rho+ \mu+\zeta $, where  $\pi\in \operatorname{Hom}(\wedge^3\mathfrak{g},\mathfrak{g})$, $\mu\in \operatorname{Hom}(\wedge^3\mathfrak{h},\mathfrak{h})$, $\rho\in \operatorname{Hom}(\wedge^2 \mathfrak{g}\otimes \mathfrak{h}  , \mathfrak{h})$ and  $\zeta\in \operatorname{Hom}(\wedge^2 \mathfrak{h}\otimes \mathfrak{g}  , \mathfrak{g})$.  Since $[\delta,\delta]_{\mathrm{R}}=0$, we have \[[\mu,\mu]_{\mathrm{R}}=0,\ [\pi,\pi]_{\mathrm{R}}=0,\ \rho\circ \rho=\rho\circ\pi \text{ and } \zeta\circ \zeta=\zeta\circ\mu.\] By Theorem~\ref{Thm: graded Lie structure of 3-Lie complex}, $\mu$ and $\pi$ define 3-Lie brackets on $\mathfrak{h}$ and $\mathfrak{g}$, respectively. $\rho\circ \rho=\rho\circ\pi$ is equivalent to that $\rho$ and $\pi$ satisfy the Equations~(\ref{Eq: 3-Lie representation1})(\ref{Eq: 3-Lie representation2}). Therefore,  $ \rho: \wedge^2 \mathfrak{g}\otimes \mathfrak{h} \rightarrow  \mathfrak{h}$ is action of   $\left(\mathfrak{g},\pi\right)$ on   $\left(\mathfrak{h},\mu\right)$. Similarly, $ \mu: \wedge^2 \mathfrak{h}\otimes \mathfrak{g} \rightarrow  \mathfrak{g}$ is action of $\left(\mathfrak{h},\mu\right)$   on   $\left(\mathfrak{g},\pi\right)$.
	
	Vice versa.
\end{proof}

By Theorem~\ref{v-th} and Proposition~\ref{Prop:subalgebra}, we can construct the following $L_{\infty}$-algebra.
\begin{prop}\label{Prop: L
		infinty on relative modified RB Lie alge}
	Let $\mathfrak{g}$ and $\mathfrak{h}$ are two vector space. Then we have a $V$-data $(V, \mathfrak{a}, \mathcal{P}, \Delta)$ as follows:
	\begin{itemize}
		\item [(a)]the graded Lie algebra $(V,[\cdot, \cdot])$ is given by $\left(C^*_{\mathrm{3Lie}}(\mathfrak{g} \oplus \mathfrak{h}, \mathfrak{g} \oplus \mathfrak{h}),[\cdot, \cdot]_{\mathrm{R}}\right)$;
		\item [(b)] the abelian graded Lie subalgebra $\mathfrak{a}$ is given by
		
		$$
		\mathfrak{a}=C^*_{\mathrm{3Lie}}(\mathfrak{h}, \mathfrak{g})=\oplus_{n \geq 0} C^n_{\mathrm{3Lie}}(\mathfrak{h}, \mathfrak{g})=\bigoplus_{n \geq 0} \operatorname{Hom}\left(\underbrace{\wedge^2 \mathfrak{h} \otimes \cdots \otimes \wedge^2 \mathfrak{h}}_{n \text{ times}} \wedge \mathfrak{h}, \mathfrak{g}\right);
		$$
		\item [(c)] $\mathcal{P}: L \rightarrow L$ is the projection onto the subspace $\mathfrak{a}$;
		\item [(d)] $\Delta=0\in \mathrm{ker}(\mathcal{P})^1$.

	\end{itemize}
	Set $V_{r\mathrm{Pair}}=C^*_{\mathrm{3Lie-act}}(\mathfrak{h},\mathfrak{g})$,  we obtain  an $L_{\infty}$-algebra $(V_{r\mathrm{Pair}}[1] \oplus \mathfrak{a}, \{l_k\}_{k\geq 1} )$, where
	
	$$
	\begin{aligned}
		&	l^{m\mathrm{RB}}_2(f[1]\otimes g[1])	=[f,g]_{\mathrm{R}}[1]\\
		&l^{m\mathrm{RB}}_k\left(f[1]\otimes \theta_1\otimes \ldots\otimes \theta_{k-1}\right)=\mathcal{P}\left[\cdots\left[\left[f, \theta_1\right]_{\mathrm{R}}, \theta_2\right]_{\mathrm{R}}, \ldots, \theta_{k-1}\right]_{\mathrm{R}}, \text { for } k \geq 2
	\end{aligned}
	$$
	
	for all $\theta_1,\ldots,\theta_{k-1}\in \mathfrak{a}$ and $f,g\in V_{r\mathrm{Pair}}$.

\end{prop}


\begin{theorem}\label{Thm: MC of modified relative RB}
	With all the above notations.     Suppose  there are maps $\pi\in \operatorname{Hom}(\wedge^3\mathfrak{g},\mathfrak{g})$, $\mu\in \operatorname{Hom}(\wedge^3\mathfrak{h},\mathfrak{h})$, $\rho\in \operatorname{Hom}(\wedge^2 \mathfrak{g}\otimes \mathfrak{h}  , \mathfrak{h})$, $\zeta\in \operatorname{Hom}(\wedge^2 \mathfrak{h}\otimes \mathfrak{g}  , \mathfrak{g})$, $T\in \operatorname{Hom}(  \mathfrak{h}  , \mathfrak{g})$ and $\delta=\pi+\rho+\lambda(\mu+\zeta)$ with nonzero $\lambda$.
	
	Then $(\delta[1], T)\in (V_{r\mathrm{Pair}}[1] \oplus \mathfrak{a})^0$ is a  Maurer-Cartan element in $L_\infty[1]$-algebra $(V_{r\mathrm{Pair}}[1] \oplus \mathfrak{a}, \{l^{m\mathrm{RB}}_k\}_{k\geq 1} )$ if and only  $\left( \left(\mathfrak{h},\pi\right),\left(\mathfrak{g},\mu\right),\rho,\zeta,T\right) $  is   a relative modified Rota-Baxter $3$-Lie algebra structure.
\end{theorem}

\begin{proof}
	Let $(\delta[1], T)\in (V_{r\mathrm{Pair}})^0$. Then $\delta$ can be decomposed
	as $\delta=\pi+\rho+\lambda(\mu+\zeta) $ and we have that
	\begin{align*}
		&	l^{m\mathrm{RB}}_2\big( (\delta[1]\otimes T), (\delta [1]\otimes T)  \big) =\left( [\delta, \delta ]_{\mathrm{R}} [1], 2\mathcal{P}[\delta,T]_{\mathrm{R}} \right) ,\\
		&	l^{m\mathrm{RB}}_2\big( (\delta[1],T)\otimes (\delta [1], T)  \big)\\
		&=\left(0, 24\mathcal{P}[[[\delta,T]_{\mathrm{R}},T]_{\mathrm{R}},T]_{\mathrm{R}}\right),
	\end{align*}
	where for all $(u,v,w)\in \wedge^2 \mathfrak{h}$
	\begin{align*}
		&\mathcal{P}[\delta,T]_{\mathrm{R}} (u,v,w)\\
		&	=-\lambda T\mu(u,v,w)+\lambda\left( \zeta(u,v)Tw+\zeta(v,w)Tu+\zeta(w,u)Tv\right) , \\
		& \mathcal{P}[[[\delta,T]_{\mathrm{R}},T]_{\mathrm{R}},T]_{\mathrm{R}}(u,v,w)\\
		&=  \pi(Tu,Tv,Tw)-T\left( \rho(Tu,Tv)w+\rho(Tv,Tw)u+\rho(Tw,Tu)v\right).
	\end{align*}
	
	Then, the element $(\delta[1], T)\in (V_{r\mathrm{Pair}})^0$ is  a  Maurer-Cartan element in  $(V_{r\mathrm{Pair}}[1] \oplus \mathfrak{a}, \{l^{m\mathrm{RB}}_k\}_{k\geq 1} )$, if and only if
	\begin{align*}
		& \sum_{k=1}^\infty \frac{1}{k !} ~\! l^{m\mathrm{RB}}_k \Big((\delta[1], T)\otimes \cdots\otimes(\delta[1], T)\Big)\\
		&=\frac{1}{2!}l^{m\mathrm{RB}}_2\big( (\delta[1], T)\otimes (\delta [1]\otimes T)+\frac{1}{4!}	l^{m\mathrm{RB}}_4\big( (\delta[1], T)\otimes (\delta [1], T)\otimes(\delta [1], T), (\delta [1], T) \big)\\
		&=0,
	\end{align*}
	that is,
	\begin{align}\label{Eq: MC element of relative RBm}
		\Big(\frac{1}{2!} [\delta, \delta ]_{\mathrm{R}}[1] ,\frac{1}{2!}\mathcal{P}[\delta,T]_{\mathrm{R}}+ \frac{1}{4!} \mathcal{P}[[[\delta,T]_{\mathrm{R}},T]_{\mathrm{R}},T]_{\mathrm{R}}\Big) =0.
	\end{align}
	By the first component of Equation~\ref{Eq: MC element of relative RBm}, we have
	\begin{align*}
		[\delta, \delta ]_{\mathrm{R}}=0.
	\end{align*}
	Thus, by Proposition~\ref{Prop:subalgebra}, this implies that the quadruple $\left( \left(\mathfrak{h},\mu\right),\left(\mathfrak{g},\pi\right),\rho,\zeta \right) $ defines a relative $3$-Lie algebra pair  structure.
	
	From the section component, we obtain
	\begin{align*}
		\pi(Tu,Tv,Tw)=&T\left( \rho(Tu,Tv)w+\rho(Tv,Tw)u+\rho(Tw,Tu)v\right) \\
		&\lambda\left( -\zeta(u,v)Tw-\zeta(v,w)Tu-\zeta(w,u)Tv+T\mu(u,v,w)\right) ,
	\end{align*}
	which shows that $T$ is a relative modified Rota-Baxter operator of weight $\lambda$ on the relative $3$-Lie algebra pair   $\left( \left(\mathfrak{h},\mu\right),\left(\mathfrak{g},\pi\right),\rho,\zeta \right) $.
	
	Therefore,  the element   $(\delta[1], T)\in (V_{r\mathrm{Pair}})^0$ is  a  Maurer-Cartan element if and only if  $( \left(\mathfrak{h},\pi\right),\left(\mathfrak{g},\mu\right),\\ \rho,\zeta, T) $  is   defines a relative modified Rota-Baxter $3$-Lie algebra structure.
\end{proof}

\subsection{$L_\infty[1]$-structure for (absolute) modified Rota-Baxter 3-Lie algebras}\label{L infty  structure for  modified Rota-Baxter 3-Lie algebras}
In this subsection, we will construct an $L_\infty[1]$-structure  whose Maurer-Cartan elements are in bijection with the set of structures of  modified Rota-Baxter $3$-Lie algebras of weight~$\lambda$. The cohomology restricted to the operator space induced by the Maurer-Cartan element of this $L_\infty[1]$-algebra constructed via the twisting procedure in Proposition~\ref{Twisting procedure} coincide with the cohomologies associated to the modified Rota-Baxter operator, as described in Theorem~\ref{Thm: cohomology of modified RB 3 Lie}.


Let $\mathfrak{g}$ be a vector space. Consider another copy of $\mathfrak{g}$, denoted by $\mathfrak{g}'$, and set $\mathfrak{h} := \mathfrak{g}'$. Define the following cochain complexes:
\[
L := C^*_{\mathrm{3Lie}}(\mathfrak{g} \oplus \mathfrak{g}',\, \mathfrak{g} \oplus \mathfrak{g}'), \qquad
V_{r\mathrm{Pair}} := C^*_{\mathrm{3Lie\text{-}act}}(\mathfrak{g}',\, \mathfrak{g}),
\]
\[
\mathfrak{a} := C^*_{\mathrm{3Lie}}(\mathfrak{g}',\, \mathfrak{g})
= \bigoplus_{n \geq 0} \operatorname{Hom}\left(
\underbrace{\wedge^2 \mathfrak{g}' \otimes \cdots \otimes \wedge^2 \mathfrak{g}'}_{n \text{ times}} \wedge \mathfrak{g}',\, \mathfrak{g}
\right).
\]

As shown in Subsection~\ref{L infty  structure for relative  modified Rota-Baxter 3-Lie algebras}, the graded vector space
\[
V_{r\mathrm{Pair}}[1] \oplus \mathfrak{a}
\]
admits an $L_\infty[1]$-algebra structure, with higher operations denoted by $\{l^{m\mathrm{RB}}_k\}_{k \geq 1}$. More precisely, we have the decomposition
\[
V_{r\mathrm{Pair}} = \mathfrak{V}_1 \oplus \mathfrak{V}_2 \oplus \mathfrak{V}_3 \oplus \mathfrak{V}_4,
\]
where:
\begin{align*}
	\mathfrak{V}_1 &= \bigoplus_{n \geq 0} \mathfrak{V}_1(n), \quad
	\mathfrak{V}_1(n+1) = \bigoplus_{i_1, \ldots, i_n \in \{0,1\}}
	\operatorname{Hom}\left(V_{i_1} \otimes \cdots \otimes V_{i_n} \otimes \wedge^3\mathfrak{g},\, \mathfrak{g} \right), \\
	\mathfrak{V}_2 &= \bigoplus_{n \geq 0} \mathfrak{V}_2(n), \quad
	\mathfrak{V}_2(n+1) = \bigoplus_{i_1, \ldots, i_n \in \{0,1\}}
	\operatorname{Hom}\left(V_{i_1} \otimes \cdots \otimes V_{i_n} \otimes \wedge^2\mathfrak{g}' \otimes \mathfrak{g},\, \mathfrak{g} \right), \\
	\mathfrak{V}_3 &= \bigoplus_{n \geq 0} \mathfrak{V}_3(n), \quad
	\mathfrak{V}_3(n+1) = \bigoplus_{i_1, \ldots, i_n \in \{0,1\}}
	\operatorname{Hom}\left(V_{i_1} \otimes \cdots \otimes V_{i_n} \otimes \wedge^3\mathfrak{g}',\, \mathfrak{g}' \right), \\
	\mathfrak{V}_4 &= \bigoplus_{n \geq 0} \mathfrak{V}_4(n), \quad
	\mathfrak{V}_4(n+1) = \bigoplus_{i_1, \ldots, i_n \in \{0,1\}}
	\operatorname{Hom}\left(V_{i_1} \otimes \cdots \otimes V_{i_n} \otimes \wedge^2\mathfrak{g} \otimes \mathfrak{g}',\, \mathfrak{g}' \right),
\end{align*}
with $V_0 := \wedge^2\mathfrak{g}'$ and $V_1 := \wedge^2\mathfrak{g}$.


Denote
\[
V_{\mathrm{Pair}} :=\bigoplus_{n \geq 0}V_{\mathrm{Pair}}(n) \text{ and } V_{\mathrm{Pair}}(n):= \operatorname{Hom}\left(\underbrace{\wedge^2 \mathfrak{g} \otimes \cdots \otimes \wedge^2 \mathfrak{g}}_{n \text{ times}} \wedge \mathfrak{g}, \mathfrak{g}\right).
\]

Consider the embedding map of graded spaces, defined for each $n \geq 1$ by
\begin{align*}
	\nu\colon V_{\mathrm{Pair}}(n)
	&\longrightarrow V_{r\mathrm{Pair}}(n) = \mathfrak{V}_1(n) \oplus \mathfrak{V}_2(n) \oplus \mathfrak{V}_3(n) \oplus \mathfrak{V}_4(n),  \\
	f
	&\longmapsto \sum_{i_1,\ldots,i_{n-1}, i_{n}, j \in \{0,1\}} f_{i_1,\ldots,i_{n-1}, i_{n}}^j,
\end{align*}
where the components are defined as follows:

For $x_1 \wedge y_1 \in V_{i_1}, \ldots, x_{n-1} \wedge y_{n-1} \in V_{i_{n-1}}$ and:
\begin{itemize}
	\item [(i)]$x, y, z \in \mathfrak{g}$, define $f_{i_1,\ldots,i_{n-1},1}^1 \in \mathfrak{V}_1(n)$ by
	\[
	f_{i_1,\ldots,i_{n-1},1}^1(x_1 \wedge y_1 \otimes \cdots \otimes x_n \wedge y_n \otimes x \wedge y \wedge z) := f(x_1 \wedge y_1 \otimes \cdots \otimes x_n \wedge y_n \otimes x \wedge y \wedge z),
	\]
	\item [(ii)]$x, y \in \mathfrak{g}'$ and $z \in \mathfrak{g}$, define $f_{i_1,\ldots,i_{n-1},0}^1 \in \mathfrak{V}_2(n)$ by
	\[
	f_{i_1,\ldots,i_{n-1},0}^1(x_1 \wedge y_1 \otimes \cdots \otimes x_n \wedge y_n \otimes x \wedge y \otimes z) := f(x_1 \wedge y_1 \otimes \cdots \otimes x_n \wedge y_n \otimes x \wedge y \wedge z),
	\]
	\item [(iii)]$x, y, z \in \mathfrak{g}'$, define $f_{i_1,\ldots,i_{n-1},0}^0 \in \mathfrak{V}_3(n)$ by
	\[
	f_{i_1,\ldots,i_{n-1},0}^0(x_1 \wedge y_1 \otimes \cdots \otimes x_n \wedge y_n \otimes x \wedge y \wedge z) := f(x_1 \wedge y_1 \otimes \cdots \otimes x_n \wedge y_n \otimes x \wedge y \wedge z),
	\]
	\item [(iv)]$x, y \in \mathfrak{g}$ and $z \in \mathfrak{g}'$, define $f_{i_1,\ldots,i_{n-1},1}^0 \in \mathfrak{V}_4(n)$ by
	\[
	f_{i_1,\ldots,i_{n-1},1}^0(x_1 \wedge y_1 \otimes \cdots \otimes x_n \wedge y_n \otimes x \wedge y \otimes z) := f(x_1 \wedge y_1 \otimes \cdots \otimes x_n \wedge y_n \otimes x \wedge y \wedge z).
	\]
\end{itemize}

	\begin{prop}\label{Prop: monomorphism of graded Lie alg}
The embedding map 	$\nu	 \colon V_{\mathrm{Pair}}
 \rightarrow V_{r\mathrm{Pair}}$  is a monomorphism of graded Lie algebra.
	\end{prop}
	
	\begin{proof}
	For \(f \in V_{\mathrm{Pair}}(n)\) and \(g \in V_{\mathrm{Pair}}(m)\), we compute:
	\begin{align*}
		[\nu(f), \nu(g)]_{\mathrm{R}}
		&= \left[ \sum_{i_1,\ldots,i_n,\, i \in \{0,1\}} f_{i_1,\ldots,i_n}^i,\
		\sum_{j_1,\ldots,j_m,\, j \in \{0,1\}} g_{j_1,\ldots,j_m}^j \right]_{\mathrm{R}} \\
		&= \sum_{i_1,\ldots,i_{m+n},\, i, j \in \{0,1\}}
		\left( f_{i_1,\ldots,i_n}^i \circ g_{i_{n+1},\ldots,i_{m+n}}^j
		- (-1)^{mn} g_{i_1,\ldots,i_m}^i \circ f_{i_{m+1},\ldots,i_{m+n}}^j \right),
	\end{align*}
	where, for each \(i_1, \ldots, i_{m+n-1} \in \{0,1\}\), the components lie in the following spaces:
	\begin{small}
\begin{align*}
		\sum_{i \in \{0,1\}} \Big(
		&f_{i_1,\ldots,i_n}^i \circ g_{i_{n+1},\ldots,i_{m+n-1},1}^1
		- (-1)^{mn} g_{i_1,\ldots,i_m}^i \circ f_{i_{m+1},\ldots,i_{m+n-1},1}^1 \Big)
		\in \operatorname{Hom}\left( V_{i_1} \otimes \cdots \otimes V_{i_{m+n-1}} \otimes \wedge^3 \mathfrak{g},\, \mathfrak{g} \right), \\
		\sum_{i \in \{0,1\}} \Big(
		&f_{i_1,\ldots,i_n}^i \circ g_{i_{n+1},\ldots,i_{m+n-1},0}^1
		- (-1)^{mn} g_{i_1,\ldots,i_m}^i \circ f_{i_{m+1},\ldots,i_{m+n-1},0}^1 \Big)
		\in \operatorname{Hom}\left( V_{i_1} \otimes \cdots \otimes V_{i_{m+n-1}} \otimes \wedge^2 \mathfrak{g}' \otimes \mathfrak{g},\, \mathfrak{g} \right), \\
		\sum_{i \in \{0,1\}} \Big(
		&f_{i_1,\ldots,i_n}^i \circ g_{i_{n+1},\ldots,i_{m+n-1},0}^0
		- (-1)^{mn} g_{i_1,\ldots,i_m}^i \circ f_{i_{m+1},\ldots,i_{m+n-1},0}^0 \Big)
		\in \operatorname{Hom}\left( V_{i_1} \otimes \cdots \otimes V_{i_{m+n-1}} \otimes \wedge^3 \mathfrak{g}',\, \mathfrak{g}' \right), \\
		\sum_{i \in \{0,1\}} \Big(
		&f_{i_1,\ldots,i_n}^i \circ g_{i_{n+1},\ldots,i_{m+n-1},1}^0
		- (-1)^{mn} g_{i_1,\ldots,i_m}^i \circ f_{i_{m+1},\ldots,i_{m+n-1},1}^0 \Big)
		\in \operatorname{Hom}\left( V_{i_1} \otimes \cdots \otimes V_{i_{m+n-1}} \otimes \wedge^2 \mathfrak{g} \otimes \mathfrak{g}',\, \mathfrak{g}' \right).
	\end{align*}
\end{small}
Thus, we have
\[
[ \nu(f) , \nu(g)]_{\mathrm{R}}  = \nu\left( f\circ g - (-1)^{mn} g\circ f \right)[1] = \nu \left( [f,g]_{\mathrm{R}} \right),
\]
that is,
\[
\nu \colon V_{\mathrm{Pair}} \rightarrow V_{r\mathrm{Pair}}
\]
is a monomorphism of graded Lie algebras.
	\end{proof}

By Theorem~\ref{v-th} and Proposition~\ref{Prop: monomorphism of graded Lie alg}, we have the following Proposition:
\begin{prop}
	\[
	\left( V_{\mathrm{Pair}}[1]\oplus \mathfrak{a}, \{l^{\mathrm{ab}}_k\}_{k \geq 1}\right)
	\]
	is an $L_\infty[1]$-algebra with operations
	\[
	\begin{aligned}
		& l^{\mathrm{ab}}_2(f[1] \otimes g[1]) = \nu^{-1}[\nu(f), \nu(g)]_{\mathrm{R}}[1], \\
		& l^{\mathrm{ab}}_k\left(f[1] \otimes \theta_1 \otimes \cdots \otimes \theta_{k-1} \right)
		= \mathcal{P} \left[ \cdots \left[ \left[ \nu(f), \theta_1 \right]_{\mathrm{R}}, \theta_2 \right]_{\mathrm{R}}, \ldots, \theta_{k-1} \right]_{\mathrm{R}}, \quad \text{for } k \geq 2,
	\end{aligned}
	\]
	for all $\theta_1, \ldots, \theta_{k-1} \in \mathfrak{a}$ and $f, g \in V_{\mathrm{Pair}}$.
	
	Moreover, $\nu$ induces a monomorphism of $L_\infty[1]$-algebras:
	\[
	\begin{aligned}
		\tilde{\nu}: \left( V_{\mathrm{Pair}}[1] \oplus \mathfrak{a}, \{l^{\mathrm{ab}}_k\}_{k \geq 1} \right)
		& \longrightarrow \left( V_{r\mathrm{Pair}}[1] \oplus \mathfrak{a}, \{l^{m\mathrm{RB}}_k\}_{k \geq 1} \right), \\
		(f[1], \theta) & \longmapsto (\nu(f)[1], \theta).
	\end{aligned}
	\]
	
\end{prop}

\begin{theorem}\label{Thm: MC of modified absolut RB}
	With all the above notations.     Suppose  there are maps $\pi\in \operatorname{Hom}(\wedge^3\mathfrak{g},\mathfrak{g})$,  $ \mathfrak{g})$, $T\in \operatorname{Hom}(  \mathfrak{g}  , \mathfrak{g})$ and $\lambda$ is non-zero square.
	
	Then $(\pi[1],\lambda^{-1/2}T)\in (V_{\mathrm{Pair}}[1] \oplus \mathfrak{a})^0$ is a  Maurer-Cartan element in $L_\infty[1]$-algebra $(V_{\mathrm{Pair}}[1] \oplus \mathfrak{a}, \{l^{\mathrm{ab}}_k\}_{k\geq 1} )$ if and only  $\left( \left(\mathfrak{g},\pi\right),\rho,\zeta,T\right) $  is   a  modified Rota-Baxter $3$-Lie algebra structure.
\end{theorem}

\begin{proof}
	Let $(\pi[1], \lambda^{-1/2}T) \in (V_{r\mathrm{Pair}})^0$. Then we have
	\begin{align*}
		l^{\mathrm{ab}}_2\big( (\pi[1], \lambda^{-1/2}T) \otimes (\pi[1], \lambda^{-1/2}T) \big)
		&= \left( \nu^{-1}[\nu(\pi), \nu(\pi)]_{\mathrm{R}}[1],\, 2\lambda^{-1/2} \mathcal{P}[\nu(\pi), T]_{\mathrm{R}} \right), \\
		l^{\mathrm{ab}}_4\big( (\nu(\pi)[1], \lambda^{-1/2}T)\otimes \cdots\otimes (\nu(\pi)[1], \lambda^{-1/2}T) \big)
		&= \left( 0,\, 24\lambda^{-3/2} \mathcal{P}[[[\nu(\pi), T]_{\mathrm{R}}, T]_{\mathrm{R}}, T]_{\mathrm{R}} \right).
	\end{align*}
	
	For all $(u,v,w) \in \wedge^2 \mathfrak{h}$, we compute:
	\begin{align*}
		\mathcal{P}[\nu(\pi), T]_{\mathrm{R}}(u,v,w)
		&= - T\pi(u,v,w)
		+ \left( \pi(u,v,Tw) + \pi(v,w,Tu) + \pi(w,u,Tv) \right), \\
		\mathcal{P}[[[\nu(\pi),T]_{\mathrm{R}},T]_{\mathrm{R}},T]_{\mathrm{R}}(u,v,w)
		&=  \pi(Tu,Tv,Tw)
		- T\left( \pi(Tu,Tv,w) + \pi(Tv,Tw,u) + \pi(Tw,Tu,v) \right).
	\end{align*}
	
	Therefore, the element $(\pi[1], \lambda^{-1/2}T) \in (V_{r\mathrm{Pair}})^0$ is a Maurer-Cartan element in the $L_\infty[1]$-algebra $(V_{\mathrm{Pair}}[1] \oplus \mathfrak{a}, \{l^{\mathrm{ab}}_k\}_{k\geq 1})$ if and only if
	\begin{align*}
		\sum_{k=1}^\infty \frac{1}{k!}~ l^{\mathrm{ab}}_k \left( (\pi[1], \lambda^{-1/2}T)^{\otimes k} \right)
		&= \frac{1}{2} l^{\mathrm{ab}}_2\left( (\pi[1], \lambda^{-1/2}T)^{\otimes 2} \right)
		+ \frac{1}{24} l^{\mathrm{ab}}_4\left( (\pi[1], \lambda^{-1/2}T)^{\otimes 4} \right) = 0.
	\end{align*}
	That is,
	\[
	[\pi, \pi]_{\mathrm{R}} = 0,
	\]
	and
	\begin{align*}
		\pi(Tu,Tv,Tw)
		&= T\left( \pi(Tu,Tv,w) + \pi(Tv,Tw,u) + \pi(Tw,Tu,v) \right) \\
		&\quad + \lambda\left( -\pi(u,v,Tw) - \pi(v,w,Tu) - \pi(w,u,Tv) + T\pi(u,v,w) \right).
	\end{align*}
	
	Hence, the element $(\pi[1], \lambda^{-1/2}T) \in (V_{r\mathrm{Pair}})^0$ is a Maurer-Cartan element if and only if the triple $\left( \mathfrak{g}, \pi, T \right)$ defines a modified Rota-Baxter 3-Lie algebra structure.
\end{proof}

\begin{remark}
Let \( R \) be a modified Rota-Baxter operator of weight \( \lambda \) on a $3$-Lie algebra \( (\mathfrak{g}, \pi) \). Then the element
\[
\alpha := (\pi[1],\, \lambda^{-1/2} T)
\]
is a Maurer-Cartan element in the \( L_\infty[1] \)-algebra \( (V_{\mathrm{Pair}}[1] \oplus \mathfrak{a}, \{l^{\mathrm{ab}}_k\}_{k \geq 1}) \). The cohomology of the cochain complex induced by \( \alpha \) via the twisting procedure described in Proposition~\ref{Twisting procedure} (restricted to \( \mathfrak{a} \)) coincides with the cohomology of the cochain complex defined in Theorem~\ref{Thm: cohomology of modified RB 3 Lie}.

\end{remark}

\subsection{Comparison of the controlling $L_\infty[1]$-algebra structure for the  relative modified Rota-Baxter $3$-Lie algebras  and the relative Rota-Baxter $3$-Lie algebras}

Hou, Sheng, and Zhou~\cite{HSZ23} constructed the controlling $L_\infty[1]$-algebra structure for relative Rota-Baxter operators of weight~$\lambda$ on $3$-Lie algebras. In this subsection, we will construct a controlling $L_\infty$-subalgebra of the $L_\infty[1]$-algebra introduced in Subsection~\ref{L infty structure for relative modified Rota-Baxter 3-Lie algebras}, which controls the deformation  of relative modified Rota-Baxter $3$-Lie algebras. This subalgebra simultaneously controls the deformations of both the relative Rota-Baxter operators and the underlying $3$-Lie algebra structures.

\begin{prop}
	Let $V_{\mathfrak{h}}$
	be a subspace of \( V_{r\mathrm{Pair}} \), and let \( \iota: V_{\mathfrak{h}} \hookrightarrow V_{r\mathrm{Pair}} \) be the natural embedding. Then \( \iota \) induces an $L_\infty$-subalgebra structure \( \{l^{\mathrm{RB}}_k\}_{k \geq 1} \) on \( V_{\mathfrak{h}}[1] \oplus \mathfrak{a} \), as well as an $L_\infty$-morphism
	\begin{align*}
		\tilde{\iota}: \big(V_{\mathfrak{h}}[1] \oplus \mathfrak{a},\{l^{\mathrm{RB}}_k\}_{k \geq 1} \big) &\longrightarrow \big(V_{r\mathrm{Pair}}[1] \oplus \mathfrak{a},\{l^{m\mathrm{RB}}_k\}_{k \geq 1} \big)\\
		(f[1],\theta)&\longrightarrow (\iota(f)[1],\theta),
	\end{align*}
	where
	
	$$
	\begin{aligned}
		&	l^{\mathrm{RB}}_2(f[1]\otimes g[1])	=\iota^{-1}[\iota(f),\iota(g)]_{\mathrm{R}}[1]\\
		&	l^{\mathrm{RB}}_k\left(f[1]\otimes \theta_1\otimes \ldots\otimes \theta_{k-1}\right)=\mathcal{P}\left[\cdots\left[\left[\iota(f), \theta_1\right]_{\mathrm{R}}, \theta_2\right]_{\mathrm{R}}, \ldots, \theta_{k-1}\right]_{\mathrm{R}}, \text { for } k \geq 2
	\end{aligned}
	$$
	
	for all $\theta_1,\ldots,\theta_{k-1}\in \mathfrak{a}$ and $f,g\in V_{\mathfrak{h}}$.
	
\end{prop}

\begin{proof}
	By the proof of Proposition~\ref{L infty  structure for relative  modified Rota-Baxter 3-Lie algebras}, we have that $V_{\mathfrak{h}}$ is also a subalgebra of the graded Lie algebra $V_{r\mathrm{Pair}}$.  Thus, by Theorem~\ref{v-th}, the bracket $\iota^{-1}[\iota(-),\ \iota(-)]$ is well-defined, and the map
	\[
	\tilde{\iota} \colon \left(V_{\mathfrak{h}}[1] \oplus \mathfrak{a},\, \{l^{\mathrm{RB}}_k\}_{k \geq 1} \right)
	\longrightarrow
	\left(V_{r\mathrm{Pair}}[1] \oplus \mathfrak{a},\, \{l^{m\mathrm{RB}}_k\}_{k \geq 1} \right)
	\]
	is an $L_\infty$-morphism.
\end{proof}

\begin{theorem}\label{Thm: MC of relative RB}
	With all the above notations.     Suppose  there are maps $\pi\in \operatorname{Hom}(\wedge^3\mathfrak{g},\mathfrak{g})$, $\mu\in \operatorname{Hom}(\wedge^3\mathfrak{h},\mathfrak{h})$, $\rho\in \operatorname{Hom}(\wedge^2 \mathfrak{g}\otimes \mathfrak{h}  , \mathfrak{h})$, $T\in \operatorname{Hom}(  \mathfrak{h}  , \mathfrak{g})$ and $\delta=\pi+\rho+\lambda\mu$ with nonzero weight $\lambda$.
	
	Then $(\delta[1], T)\in (V_{\mathfrak{h}}[1] \oplus \mathfrak{a})^0$ is a  Maurer-Cartan element in $L_\infty[1]$-algebra $(V_{\mathfrak{h}}[1] \oplus \mathfrak{a}, \{l^{\mathrm{RB}}_k\}_{k\geq 1} )$ if and only  $\left( \left(\mathfrak{h},\pi\right),\left(\mathfrak{g},\mu\right),\rho,T\right) $  is   a relative  Rota-Baxter $3$-Lie algebra structure.
\end{theorem}

\begin{proof}
	As in the proof of Theorem~\ref{Thm: MC of modified relative RB}, we omit the details here.
\end{proof}

\begin{remark}
	Let $\alpha=\left( \left(\mathfrak{h},\pi\right),\left(\mathfrak{g},\mu\right),\rho, T\right)$ be a relative Rota-Baxter $3$-Lie algebra with trivial $T$. By Theorem~\ref{Thm: MC of relative RB}, $\alpha$ is also a Maurer-Cartan element in the $L_\infty[1]$-algebra $(V_{\mathfrak{h}}[1] \oplus \mathfrak{a}, \{l^{\mathrm{RB}}_k\}_{k\geq 1})$.
	
	By Proposition~\ref{Twisting procedure}, we obtain the twisted $L_\infty[1]$-algebra structure $\{(l^{\mathrm{RB}})^{\alpha}_k\}_{k\geq 1}$ on $V_{\mathfrak{h}}[1] \oplus \mathfrak{a}$ induced by the Maurer-Cartan element $\alpha$. In fact, the subalgebra $\mathfrak{a}$ of the twisted $L_\infty[1]$-algebra $(V_{\mathfrak{h}}[1] \oplus \mathfrak{a}, \{(l^{\mathrm{RB}})^{\alpha}_k\}_{k\geq 1})$ coincides with the $L_\infty[1]$-algebra introduced in Section~4 of \cite{HSZ23} by Hou, Sheng, and Zhou. This $L_\infty[1]$-algebra controls the deformation theory of relative Rota-Baxter operators of weight~$\lambda$ on $3$-Lie algebras $\left(\mathfrak{h},\pi\right)$, $\left(\mathfrak{g},\mu\right)$ together with the action $\rho$.
	
\end{remark}

\bigskip
\noindent
{\bf Acknowledgments.} This work  was   supported by  the National Key R$\&$D Program of China (No. 2024YFA1013803), by Natural Science Foundation of China (No. 12161013),      by the project OZR3762 of Vrije Universiteit Brussel, by the FWO Senior Research Project G004124N,    by Shanghai Key Laboratory of PMMP (No. 22DZ2229014), and by Guizhou Provincial Basic Research Program (Natural Science)(No. MS[2026]996). 

\noindent
\textbf{Author contributions}
All authors contributed equally.

\noindent
{ \bf Data availability} Data sharing not applicable to this article as no datasets were
generated or analysed during the current study.

\noindent
{ \bf Declarations}  \\
{ \bf Conflict of interest}  The authors have no relevant financial or non-financial
interests to disclose.

\noindent
\textbf{Ethical approval}  Not applicable.

\end{document}